\documentclass[12pts]{article}
\usepackage{graphicx} 
\usepackage{fullpage}
\usepackage{float}
\usepackage{amssymb,amsfonts,amsthm,nicefrac,latexsym,amsfonts,amsbsy,bm, amscd, mathtools}
\usepackage{enumitem}
\usepackage{todonotes}
\usepackage{hyperref}
\hypersetup{
    colorlinks=false
}
\usepackage{multirow}

\usepackage{yhmath}

\DeclareMathOperator{\sech}{sech}
\DeclareMathOperator{\diag}{diag}
\DeclareMathOperator{\interior}{int}

\newtheorem{remark}{Remark}
\newtheorem{theorem}{Theorem}[section]
\newtheorem{corollary}[theorem]{Corollary}
\newtheorem{lemma}[theorem]{Lemma}
\newtheorem{definition}{Definition}[section]

\newtheorem*{rmk_main}{Remark on the conditions of Theorem \ref{thm:fractal_dim_H}}
\usepackage{authblk}
\title{On the dimension of pullback attractors in recurrent neural networks}
\author[1]{Muhammed Fadera$^1$}

\date{}

\begin{document}
\maketitle
\begin{abstract}
Recurrent neural networks trained via the reservoir computing paradigm have demonstrated remarkable success in learning and reconstructing attractors from chaotic systems, often replicating quantities such as Lyapunov exponents and fractal dimensions. It has recently been conjectured that this is because the reservoir computer embeds the dynamics of the chaotic system in its state space before learning. This conjecture has been established for reservoir computers with linear activation functions and remains open for more general reservoir systems. In this work, we employ a non-autonomous dynamical systems approach to establish an upper bound for the box-counting dimension of the pullback attractor, a subset of the reservoir state space that is approximated during training and prediction phases. We prove that the box-counting dimension of the pullback attractor is bounded above by the box-counting dimension of the space of input sequences with respect to the product topology. In particular, for input sequences originating from an $N_\text{in}$-dimensional smooth dynamical system or their generic continuously differentiable observations, the box-counting dimension of the pullback attractor is bounded above by $N_\text{in}$. The results obtained here highlight the fact that, while a reservoir computer may possess a very high-dimensional state space, it exhibits effective low-dimensional dynamics. Our findings also partly explain why reservoir computers are successful in tasks such as attractor reconstruction and the computation of dynamic invariants like Lyapunov exponents and fractal dimensions.  
\end{abstract}

\footnotetext[1] {
        Global Systems Institute, 
        University of Exeter, 
        Exeter, 
        EX4 4QE, 
        Devon, UK}

\section{Introduction}
\label{sec:intro}
Reservoir computing (RC) framework has offered a computationally efficient way to train recurrent neural networks by randomly generating the internal weights and keeping them fixed while focusing training only on the output layer \cite{jaeger2001echo, Maass2002}. This training procedure allows reservoir computers (RCs) to avoid the notorious vanishing and exploding gradient problem \cite{pascanu2013difficulty}, and is one of the main contenders for biologically plausible temporal learning \cite{dominey1995complex, dominey2021cortico}. Since a large class of RCs are universal approximators of fading memory mappings on left infinite input sequences \cite{grigoryeva2018echo, Gonon2021FadingUniversal}, they have achieve state-of-the-art performance on several benchmarks, including time-series classification \cite{hewamalage2021recurrent, LiESN}, systems identification \cite{jaeger2002adaptive}, adaptive filtering \cite{xia2008complex}, attractor reconstruction \cite{lu2018attractor}, anticipating tipping points and critical transitions \cite{panahi2024machine, mitsui2021seasonal}, etc.
\\
\\
Beyond the theory of universal approximations, it is unclear how RCs are able to achieve such performance on these tasks. In the past few years, efforts have been made to understand the decision-making process of such systems by studying them as input-driven dynamical systems \cite{sussillo2013opening, ceni10, flynn2021multifunctionality}. In the case of classification problems with discrete-valued inputs, it has been shown, for example, in \cite{sussillo2013opening, ceni2020interpreting} that echo state networks (ESNs) trained with output feedback take advantage of stable fixed points and their relative locations to solve such tasks. As well as in the continuous output case, errors in ESNs trained on such tasks have been identified with the emergence of spurious/untrained attractors \cite{flynn2021multifunctionality, ceni2020interpreting}. Transitions to such attractors may cause errors over short time steps in classification problems with discrete-valued inputs and over longer time steps in continuous output problems \cite{flynn2021multifunctionality}. Even with such erroneous transitions, it is known, for example, in tasks such as attractor reconstruction, that RCs can maintain the general dynamics of the system they were trained on well beyond when their predictions start to diverge. This property of RCs allows them to learn dynamic invariants, such as Lyapunov exponents and fractal dimensions, of the system they were trained on \cite{lu2018attractor, hart2020embedding, hart2024attractor, pathak2017using}.
\\
\\
Although it is known that the asymptotic dynamics of a reservoir computer synchronises with those of the driving input dynamics \cite{parlitz1996generalized}, this synchronisation may not be sufficient for explaining the performance of reservoir computers \cite{grigoryeva2021chaos}. In \cite{hart2020embedding}, the authors conjectured that reservoir computers driven with generic continuously differentiable ($C^1$) observations of a smooth dynamical system on an $m$-dimensional compact manifold embed the dynamics of the dynamical system in its state space. This conjecture has been settled in the case of linear reservoirs \cite{hart2021strange}. If the conjecture holds in the general case, it will explain why reservoir computers achieve good performance in tasks such as attractor reconstruction.
\\



In this paper, we focus on establishing bounds on the dimension of the synchronised dynamics of the reservoir when driven with a certain class of input sequences. In particular, we investigate the box-counting dimension of the pullback attractor in the nonautonomous dynamical system arising from a large class of RCs of the form 
\begin{align} \label{eq:main}
    \mathbf{x}[n+1] = g(\mathbf{u}[n+1], \mathbf{x}[n])
\end{align}
where the inputs $\mathbf{u}[n] \in U \subseteq \mathbb{R}^{N_{\text{in}}}$, the states $\mathbf{x}[n] \in X \subseteq \mathbb{R}^{N_r}$, and $g$ is some continuous function on a compact set $U$ (the space of inputs) and a compact set $X$ (the state space). We show that under mild conditions on $g$, the box-counting dimension of the pullback attractor associated with input sequences coming from a finite-dimensional subset of $U^\mathbb{Z}$ is bounded above by the box-counting dimension of this subset in the product topology. In the case that the input sequence comes from trajectories of a bi-Lipschitz homeomorphism $\phi$ on $U$, we show that the box-counting dimension of the pullback attractor is bounded above by the box-counting dimension of $U$. Remarkably, this upper bound does not require $g$ to satisfy the echo state property (ESP). However, if $g$ satisfy ESP and the set $U$ contains an open set, we show that the pullback attractor associated with all input sequences in $U^\mathbb{Z}$ also contains an open set in the reservoir space and thus has dimension $N_r$. 
\\
\\
The remainder of this paper is organised as follows. Section \ref{sec:back} describes notations, background and some preliminary results used in this paper. The space $U^\mathbb{Z}$ of input sequences is introduced in Section \ref{ssec:input_seq} together with the product topology on it. We then introduce the space of trajectories of a homeomorphism $\phi$ on $U$ as a subset of $U^\mathbb{Z}$. The dynamics of RNNs are presented in Section \ref{ssec:rnn_dynamics} and how they are trained via the reservoir computing paradigm. After introducing RNNs as nonautonomous (skew-product) dynamical systems in Section \ref{ssec:rnn_nads}, the notion of a pullback attractor and its characterisations are introduced. Section \ref{sec:dim_pull} contains the main contribution of this paper. It begins with the definition and properties of the box-counting dimension. We then give the Theorem of Vishik and Chepyzhov \cite{chepyzhov2002attractors}, providing an upper bound for the box-counting dimension of the pullback attractor for finite-dimensional input sequence spaces (Theorem \ref{thm:vishikchepyzhov}). We show in Section \ref{ssec:finite_dim_input_seqs} that if $g(u, x)$ is continuously differentiable on $U \times X$, then the corresponding nonautonomous dynamical system satisfies the conditions of Theorem \ref{thm:vishikchepyzhov}. We apply Theorem \ref{thm:vishikchepyzhov} to the space of trajectories of $\phi$ on $U$, which is shown to have box-counting dimension equal to that of $U$ (Theorem \ref{prop:dyn_euclidean_prod}). In Section \ref{ssec:input_seq_U_Z}, we shift our focus to the pullback attractor corresponding to all input sequences in $U^\mathbb{Z}$. Using a characterisation of the pullback attractor as the image of a continuous map and under additional genericity conditions on finite-time iterations of $g$, we show in Theorem \ref{thm:fractal_dim_H} that if $U$ contains an open set, the pullback attractor will also contain an open set. Consequently, it has a box-counting dimension equal to the dimension of the reservoir state space $X$. Section \ref{sssec:application_echo} studied the genericity conditions of Theorem \ref{thm:fractal_dim_H} for ESNs. We numerically verify the upper bound of Theorem \ref{thm:vishikchepyzhov} in Section \ref{sec:numerical} by driving moderate sized ESNs with trajectories of the Rössler and Lorenz attractors and estimate the box-counting dimension of the set containing all points on these trajectories after a long washout. We also verify Theorem \ref{thm:fractal_dim_H} by driving ESNs with the periodic input sequences of sufficiently large period. We conclude with some discussions and some open questions related to the work conducted here.

\section{Background} \label{sec:back}
\subsection{Notation}
Let $\mathbb{N}$ denote the set of natural numbers and $\mathbb{N}_0$ denote the set $\mathbb{N} \cup \{0\}$. We begin with some general notation on Euclidean spaces and functions defined on them. Let $m \in \mathbb{N}$ and $u, v \in \mathbb{R}^m$. We endow $\mathbb{R}^m$ with the norm
\begin{align} \label{eq:norm}
    ||u - v|| = \sum_{i = 1}^m | u_i - v_i|.
\end{align}
Suppose $V$ is a subset of $\mathbb{R}^m$. We write $\interior {V}$ to be the interior of $V$ and $\overline{V}$ the closure of $V$, which are the largest open set contained in $V$ and the smallest closed set containing $V$, respectively. Now suppose $V$ is an open subset of $\mathbb{R}^m$ and $Z$ an open subset of $\mathbb{R}^n$. Let $f: V \times Z \to Z$ be a continuously differentiable ($C^1$) function. We denote the derivative of $f$ at $(v, z) \in \mathbb{R}^m \times \mathbb{R}^n$ by $Df(v, z)$, which is a linear map from $\mathbb{R}^{m+n}$ to $\mathbb{R}^m$. The partial derivative of $f$ with respect to the $i^\text{th}$ argument is denoted $D_i f(v, z)$. If $L: \mathbb{R}^{n} \to \mathbb{R}^m$ is a linear map, we write
\begin{align*}
    ||L||_2 = \sup_{||u|| = 1} ||L(u)||
\end{align*}
as the $2$-norm of $L$, which is also the largest singular value of $L$. Obviously, $n \times m$ matrices are linear maps and every linear map from $\mathbb{R}^m$ to $\mathbb{R}^n$ has an $n \times m$ matrix $M$ associated with it \cite{apostol1974mathematical}. 

\subsection{The space of input sequences}\label{ssec:input_seq}
Before introducing the dynamics of RNNs, we will briefly explain the space of input sequences. Suppose $N_{\text{in}}  \in \mathbb{N}$  and  $U$ is compact subset of $\mathbb{R}^{N_{\text{in}}}$ where  $\mathbb{R}^{N_{\text{in}}}$ is equipped with the norm $||\cdot||$. Now consider the Cartesian product $\mathcal{U} = U^\mathbb{Z}$ consisting of all bi-infinite sequences of $N_{\text{in}}$-dimensional vectors $\mathbf{u} = \{\mathbf{u}[n]\}_{n=-\infty}^{\infty}$ where the terms $\mathbf{u}[n] \in U$ is the $n^\text{th}$ of $\mathbf{u}$ for each $n \in \mathbb{Z}$. We reserve the notation $\underline{\mathbf{u}}$ of boldface lowercase letters with an underline for elements of $U^m$ for some $m \ge 1$, while subsets of $U^m$ are denoted by boldface capital letters. Let $w:\mathbb{N}_0 \to (0, 1)$ be a decreasing weighting sequence tending to $0$. We can consider $\mathcal{U}$ as metric space with the metric 
\begin{align*}
    d^w(\mathbf{u}, \mathbf{v}) = \sup_{n \in \mathbb{Z}} w_{|n|} ||\mathbf{u}[n] - \mathbf{v}[n] ||
\end{align*}
which induces the product topology on $\mathcal{U}$, regardless of the choice of weighting sequence \cite[Theorem 2.6]{grigoryeva2018echo}. 
Additionally, $\mathcal{U}$ is compact in this topology by Tychonoff's theorem \cite[Theorem 37.3]{munkres2000topology}. 
\\
\\
We denote $C(\mathcal{U}, X)$ as the space of continuous functions from $\mathcal{U}$ to $X$. Equipped with the uniform metric 
\begin{align} \label{def:unif_metric}
    \rho(H, F) = \sup_{\mathbf{u} \in \mathcal{U}} ||H(\mathbf{u}) - F(\mathbf{u})|| \quad H, F \in C(\mathcal{U}, X), 
\end{align}
$C(\mathcal{U}, X)$ is a complete metric space \cite[Theorem 43.6]{munkres2000topology}.
\\
\\
A special class of input sequences, which is routinely used in reservoir computing, is the space of trajectories from a finite-dimensional dynamical system. If $\phi:U \to U$ is a homeomorphism, we denote the space of all trajectories
\begin{align*}
        (\ldots, \phi^{-2}(u), \phi^{-1}(u), u, \phi(u), \phi^{2}(u), \ldots), \: u \in U;
\end{align*}
as $\mathcal{V}_{(U, \phi)}$. Note that this also includes time-$T$ maps coming from solutions of ordinary differential equations. A suitable choice of  $\mathcal{V}_{(U, \phi)}$ can also be used to represent the space of real-valued observations of dynamical systems on abstract manifolds. In particular, suppose $V$ is a smooth compact manifold of dimension $m$, $\phi:V \to V$ is a diffeomorphism and $\omega: V \to \mathbb{R}$ is a $C^1$ observation function such that the $(2m+ 1)$-delay map 
\begin{align*}
     \Psi(u) = (\omega(u), \omega \circ \phi(u), \ldots, \omega \circ \phi^{2m}(u)), \: u \in V,
\end{align*}    
    is an embedding. Then the space input sequences of the form 
\begin{align*}
        (\ldots, \omega(\phi^{-2}(u)), \omega(\phi^{-1}(u)), \omega(u), \omega(\phi(u)), \omega(\phi^{2}(u)), \ldots), \: u \in V;
\end{align*}
can be represented as the set $\mathcal{V}_{(\Psi(V), \Psi \circ \phi \circ \Psi^{-1})}$. By Takens' delay-embedding theorem \cite{takens2006detecting}, the assumption that $\Psi(u)$ is an embedding holds for a generic choice of $\omega \in C^1$. 
\subsection{Dynamics of recurrent neural networks} \label{ssec:rnn_dynamics}
Given an initial condition $\mathbf{x}[0] \in \mathbb{R}^{N_r}$ with $N_r \in \mathbb{N}$, the dynamics of the internal states of an RNN with $N_r$ hidden units or $N_r$ dimensions at time-step $n$, $n \in \mathbb{N}_0$ driven by an input sequence $\mathbf{u} \in  \mathcal{U}^{}$ is given by
\begin{align}
    \mathbf{x}[n + 1] &= f{(W\mathbf{x}[n]+ W^\text{in}\mathbf{u}[n+1]+ W^\text{fb} \mathbf{y}[n])} \label{eq:rnn_state} \\
    \mathbf{y}[n+1]&= \varphi(W^\text{out} \mathbf{x}[n + 1]) \label{eq:rnn_readout}
\end{align}
where $W \in \mathbb{R}^{N_r \times N_r}$ is called the recurrent matrix, $W^\text{in} \in \mathbb{R}^{N_r \times N_\text{in}}$ is the input-to-network coupling, $W^\text{fb} \in \mathbb{R}^{N_r \times N_\text{out}}$ feeds back the output into the network and $W^\text{out} \in \mathbb{R}^{N_r \times N_\text{out}}$ is the readout weights. The functions $\varphi$ and $f$ are activation functions applied component-wise. In reservoir computing, the weights $W$, $W^\text{in}$ and $W^\text{fb}$ are initialised randomly and then kept fixed during training, while $W^\text{out}$ is optimised for the RNN to solve a particular task \cite{jaeger2001echo}. The activation function $f$ is usually taken to be the hyperbolic tangent $\tanh$ and $\varphi$ the identity. 
This architecture, known as an echo state network (ESN), is the most commonly used architecture in reservoir computing. ESNs can also be adapted for multiple time-scales \cite{LiESN, tallec2018can} using the Leaky integrator dynamics
\begin{align} \label{eq:leaky_rnn_1pdate}
    \mathbf{x}[n+1] &= (1-\alpha)\mathbf{x}[n] + \alpha f{(W\mathbf{x}[n] + W^\text{in}\mathbf{u}[n+1] + W^\text{fb} \mathbf{y}[n]}) 
\end{align}
at leak rate $\alpha$ \cite{LiESN}. Thus, the dynamics of an RNN can be written as 
\begin{align}
    \mathbf{x}[n+1] = g(\mathbf{u}[n+1], \mathbf{x}[n]).
\end{align}
for some function $g$. 
\subsection{RNNs as nonautonomous dynamical system}\label{ssec:rnn_nads}
The input-driven dynamics of RNNs are best understood in the lens of nonautonomous dynamical systems \cite{KloedenNonautonomousSystems}. There are two main ways of dealing with nonautonomous dynamical systems: as a process or a skew product dynamical system \cite{KloedenNonautonomousSystems}.  The skew product flow formulation is also closely related to the systems of processes formulation \cite{chepyzhov2002attractors}, which is sometimes easier to work with (see, for example \cite{cunha2024smoothing}). The general idea behind these two approaches is to study the combined dynamics of the nonautonomous inputs (also known as parameters) along with the states. For RNNs, the process formulation is useful when dealing with the properties of a single input sequence \cite{mane2006dimension}, while the skew product flow formulation allows more general treatment of all input sequences taking values in a compact set \cite{Ceni2020echo}. Since we are concerned with all input sequences taking values in $U$, we will use the skew product flow formulation in this paper. 
\\

The dynamics in the space of input sequences $\mathcal{U}$ is defined in terms of shifts. We define the canonical projections
\begin{align*}
\pi_{n}:\mathcal{U} &\to  U, \\
\pi_{n}(\mathbf{u}) &= \mathbf{u}[n].
\end{align*}
Let 
\begin{align*}
    \theta:\mathcal{U}& \to \mathcal{U} \\
    \{\mathbf{u}[n]\}_{n=-\infty}^{\infty} & \to \{\mathbf{u}[n+1]\}_{n=-\infty}^{\infty}
\end{align*}
be the left shift operator. The left shift operator is continuous (with respect to the product topology) since each of its projections $\pi_n(\theta(\mathbf{u})) = \mathbf{u}[n+1]$ is continuous as a function from $\mathcal{U}$ to $U$ \cite[Theorem 19.6]{munkres2000topology}. 
Furthermore, it is also invertible with inverse $\theta^{-1}$ as the right shift operator satisfying $\pi_n(\theta^{-1}(u)) = \mathbf{u}[n-1]$ which is also continuous by the preceding argument. 
Thus $\{\theta^n\}_{n \in \mathbb{Z}}$ is a discrete-time dynamical system on the state space $\mathcal{U}$ of input sequences with $\theta^0(\mathbf{u}) = \mathbf{u}$ and 
    \begin{align*}
        \theta^{m+n}(\mathbf{u}) = \theta^m \circ \theta^n(\mathbf{u}) \quad \text{for all }m, n \in \mathbb{Z}_{} \: \text{and for all }\: \mathbf{u} \in \mathcal{U}.
    \end{align*}
Now define the skew product flow on $\mathcal{U} \times X$.
\begin{definition}[Skew product flow, \cite{KloedenNonautonomousSystems}] \label{def:skewprod}
Suppose $U$ and $X$ are compact subsets of $\mathbb{R}^{N_\text{in}}$ and $\mathbb{R}^{N_r}$, respectively. Let $g: U \times X \to \mathbb{R}^{N_r}$ be a continuous map. We define the cocycle mapping $\Phi:\mathbb{N}_0^{} \times U \times X \to X$ as 
    \begin{align}
       & \Phi(0, \mathbf{u}  , x) := x \quad  \forall \mathbf{u}  \in  \mathcal{U}  \:\text{ and  }\: \forall x \in X, \label{eq:skewprod0} \\
       & \Phi(n, \mathbf{u} , x) := g(\pi_0(\theta^n(\mathbf{u})), \Phi(n-1,  \mathbf{u} , x)) = g(\mathbf{u}[n], \Phi(n-1,  \mathbf{u} , x)) \quad \forall  n \in \mathbb{N}.  \label{eq:skewprod1}
    \end{align}
    The pair $(\theta, \Phi)$ is called a skew product flow with base space $\mathcal{U}$ and state space $X$.
\end{definition}

The map $\Phi$ is called a cocycle because it satisfies the cocycle property. 

\begin{definition}[The cocycle property]\label{lmm:cocycle_property}
    The map $\Phi$ satisfies the cocycle property \textit{i.e.} for any $\mathbf{u} \in \mathcal{U}$ and $x \in X$,
\begin{align} \label{eq:cocycle_property}
    \Phi(n+m, \mathbf{u}, x) = \Phi(n, \theta^m (\mathbf{u} ), \Phi(m, \mathbf{u} , x)) \quad \text{for all } m, n \in \mathbb{N}_0 ^{}. 
\end{align}
\end{definition}

To see that $\Phi$ indeed satisfies the cocycle property, let $\mathbf{u} \in \mathcal{U}$, $x \in X$ and $m \in \mathbb{N}_0$ be fixed. Arguing by induction on $n$, we have 
    \begin{align*}
        \Phi(n,\theta^m( \mathbf{u}), \Phi(m, \mathbf{u}, x)) & =  g(\pi_0(\theta^{n}(\theta^m(\mathbf{u}))), \Phi(n-1, \theta^m( \mathbf{u}), \Phi(m, \mathbf{u}, x))) \\
        &=  g(\pi_0(\theta^{m+n}(\mathbf{u})), \Phi(n-1, \theta^m( \mathbf{u}), \Phi(m, \mathbf{u}, x))) \\
        &= g(\pi_0(\theta^{m+n}(\mathbf{u})), \Phi(m+n-1,\mathbf{u},  x)) \\
        &= \Phi(m + n, \mathbf{u},  x).
    \end{align*}
\subsubsection{Entire solutions, the Echo State Property and Pullback attraction}
The dynamics of the skew product flow $(\theta, \Phi)$ is studied in $X$ in terms of a family of subsets $\mathbf{A}_\mathbf{u}$ of $X$ indexed by $\mathbf{u}$. The set $\mathbf{A} = \{\mathbf{A}_\mathbf{u}\}_{\mathbf{u} \in \mathcal{U}}$ is known as a nonautonomous set, and $\mathbf{A}_\mathbf{u}$ the $\mathbf{u}$-fibre \cite{KloedenNonautonomousSystems}\footnote{In general, a $\mathbf{u}$-fibre of nonautonomous set can be thought of as the subset of the state space $X$ explored by $\mathbf{u}$ in relation to a studied property. For example, the trivial nonautonomous set $\mathbf{A} = \{X\}_{\mathbf{u} \in \mathcal{U}}$ describes the state space the input-driven dynamics of each input sequence live.}. A nonautonomous set $\mathbf{A}$ is said to be non-empty (respectively compact) if each of the fibres is non-empty (resp. compact). 
\begin{definition} 
A non-empty and compact nonautonomous set $\mathbf{A} = \{\mathbf{A}_\mathbf{u}\}_{\{\mathbf{u} \in \mathcal{U}\}}$ is said to be $\Phi$-invariant if 
\begin{align}
   \Phi(n, \mathbf{u}, \mathbf{A}_\mathbf{u}) = \mathbf{A}_{\theta^n(\mathbf{u})}\quad  \text{for all } \: n \in \mathbb{N}, \mathbf{u} \in \mathcal{U}.\label{eq:nonauto_invariance}
\end{align} 
\end{definition}
For each input sequence $\mathbf{u}$, a $\Phi$-invariant set is composed of terms of sequences $\{\mathbf{x}[n]\}_{n \in \mathbb{Z}}$ that solve the input-driven relation \eqref{eq:main} for all $n \in \mathbb{Z}$. These sequences are referred to as entire/complete solutions.  
\begin{definition}
 A sequence $\{\mathbf{x}[n]\}_{n \in \mathbb{Z}}$ is called an entire solution for $\mathbf{u}$ if 
\begin{align*}
    \mathbf{x}[n] = \Phi(n-m, \theta^m(\mathbf{u}), \mathbf{x}[m])\quad \text{ for all } m, n \in \mathbb{Z}, n \ge m. 
\end{align*}   
\end{definition}

Although an entire solution is guaranteed to exist in a $\Phi$-invariant set for every input sequence \cite[Chapter 2, Lemma 2.20]{KloedenNonautonomousSystems}, such a solution may not be unique. This is the case if for a fixed $u \in U$, $g(u, \cdot):X \to X$ is constant, in which case every bi-infinite sequence in $X^\mathbb{Z}$ is an entire solution for $\mathbf{u} = \{u\}_{n \in \mathbb{Z}}$. The property that every input sequence has a unique entire solution associated with it is called the echo state property (ESP). 

\begin{definition}[Echo state property]
    The skew product flow $(\theta, \Phi)$ is said to have the echo state property if for any input sequence $\mathbf{u}$, there exists a unique entire solution associated with it. 
\end{definition}
\begin{remark}
    While the ESP is a desirable property is reservoir computers, it has been shown that RNNs satisfying the related property of echo-index bigger than $1$ \cite{Ceni2020echo} admits multiple entire solutions which attracts the entire state space $X$ except possibly a small subset (in the sense of Lebesgue). 
\end{remark}



    
Pullback attractors are a particular type of $\Phi$-invariant sets that generalise the concept of attractions for non-autonomous dynamical systems. In what follows, 
\begin{align*}
    \text{dist}(A, B) := \sup_{a \in A} {\inf_{b \in B} d (a, b) }
\end{align*}
is the Hausdorff semi-metric between two non-empty subsets $A$, $B$ of a metric space $(Z, d)$. 

\begin{definition}[Pullback Attractor]
A $\Phi$-invariant set $\mathbf{A} = \{\mathbf{A}_\mathbf{u}\}_{\mathbf{u} \in \mathcal{U}}$ is called the pullback attractor of the skew-product flow $(\theta, \Phi)$ if for each input sequence $\mathbf{u} \in \mathcal{U}$ and for every bounded subset $B$ of $X$, we have
\begin{align*}
    \lim_{n \to  \infty} \normalfont \text{dist}(\Phi(n, \theta^{-n}(\mathbf{u}), B), \mathbf{A}_{\mathbf{u}}) = 0. 
\end{align*}
\end{definition}

If $g(U, X)$ is bounded, the fibres of the pullback attractor satisfy \cite[Corollary 1.17]{robinson2012attractors} 
\begin{align} \label{eq:robinson}
    \mathbf{A}_{\mathbf{u}} = \{ \mathbf{x}[0]: \mathbf{x} \text{ is an entire solution associated to } \mathbf{u}\}.
\end{align}
Hence, if a cocycle satisfies the echo state property, the $\mathbf{u}$-fibre of the pullback attractor is $\{\mathbf{x}[0]\}$ where $\mathbf{x}$ is the unique entire solution associated with $\mathbf{u}$.

\subsubsection{Pullback Attractor as the image of a continuous map} \label{ssec:pull_continuous}
In this section, we will show that if $g(u, x)$ is a contraction in $x$ $i.e$ there exist $0 < \mu < 1$ such that 
\begin{align*}
    ||g(u, x) - g(u, z)|| \le \mu ||x - z||,
\end{align*}
then the fibres of the pullback attractor can be obtained as the image of a continuous map $H$ from $\mathcal{U}$ to $\mathbb{R}^{N_r}$. This result has appeared, for example, in \cite{hart2020embedding, grigoryeva2021chaos} as the echo state mapping theorem when $g$ is driven with $C^1$ scalar observations of a diffeomorphism $\phi$ on a compact $m$-manifold $V$. The map $H$ in this case is called a generalised synchronisation (GS)  \cite{parlitz1996generalized}, and is known to be continuously differentiable under additional assumptions relating the rate of contraction $\mu$ to the derivative of $\phi^{-1}$. The GS maps the time evolution of $\phi$ on $V$ to the time evolution of (\ref{eq:main}) and as such, a desirable property of the GS is that it is a diffeomorphism onto its image (an embedding in the sense of differential topology). This ensures that the quantities, such as Lyapunov exponents and fractal dimension, are preserved. The question of whether or not the generalised synchronisation is an embedding is still open in the general case \cite{hart2020embedding} and settled for $g(u, x) = Wx + W^\text{in}\omega(u)$ \cite{hart2021strange} where $\omega:V \to \mathbb{R}$ is a generic $C^1$ scalar observation of the dynamical system, $W$ and $W^\text{in}$ are generic choices of weight matrices.  
\\
\begin{theorem} \label{thm:the_map_H}
    Suppose that the map $g:U \times X \to X$ is uniformly a $\mu$-contraction in the state variable. Let $H_0 \in C(\mathcal{U}, X)$ and $H_n$ be defined as
\begin{align} \label{eq:H_n}
    H_{n+1}(\mathbf{u}) = \Phi(1, \mathbf{u}, H_{n}(\theta^{-1}(\mathbf{u}))) \quad \text{for }\: n \in \mathbb{N}_0
    \end{align}
    
    Then there exists a unique continuous map $H \in C(\mathcal{U}, X)$ such that $\lim_{n \to \infty} H_n = H$ with respect to the uniform metric $\rho$ on $C(\mathcal{U}, X)$ for all initial choice $H_0 \in C(\mathcal{U}, X)$. 
\end{theorem}

\begin{proof}
    The proof will follow the arguments of the proof of \cite[Theorem 2.2.2]{hart2020embedding}. Let $\mathbf{\Psi}:C(\mathcal{U}, X) \to C(\mathcal{U}, X)$ be the map $H \to \Phi(1, \theta^0(\cdot), H(\theta^{-1}(\cdot))$. We will show that $\mathbf{\Psi}$ is a $\mu$-contraction on $C(\mathcal{U}, X)$. Let $F, H \in C(\mathcal{U}, X)$. Then 
    \begin{align*}
        \rho(\mathbf{\Psi}(F), \mathbf{\Psi}(H)) &= \sup_{\mathbf{u} \in \mathcal{U}} ||\Phi(1, \mathbf{u}, H(\theta^{-1}(\mathbf{u}))) -  \Phi(1, \mathbf{u}, F(\theta^{-1}(\mathbf{u})))|| \\
        &= \sup_{\mathbf{u} \in \mathcal{U}} ||g(\mathbf{u}[1], \Phi(0, \mathbf{u}, H(\theta^{-1}(\mathbf{u})))) - g(\mathbf{u}[1], \Phi(0, \mathbf{u}, F(\theta^{-1}(\mathbf{u})))|| \\
        &= \sup_{\mathbf{u} \in \mathcal{U}} ||g(\mathbf{u}[1], H(\theta^{-1}(\mathbf{u}))) - g(\mathbf{u}[1], F(\theta^{-1}(\mathbf{u})))|| \\
        &\le \sup_{\mathbf{u} \in \mathcal{U}} \mu \cdot ||H(\theta^{-1}(\mathbf{u}))- F(\theta^{-1}(\mathbf{u}))|| \\
        &= \mu \sup_{\mathbf{u} \in \mathcal{U}} ||H(\theta^{-1}(\mathbf{u}))- F(\theta^{-1}(\mathbf{u}))|| \\
        &= \mu \rho(F, H). 
    \end{align*}
    So $\mathbf{\Psi}$ is a $\mu$-contraction on $C(\mathcal{U}, X)$. Since $C(\mathcal{U}, X)$ is a complete metric space, by Banach fixed-point theorem, $\mathbf{\Psi}$ has a unique fixed point in $C(\mathcal{U}, X)$. That is, there exists a unique continuous map $H \in C(\mathcal{U}, X)$ such that $\mathbf{\Psi}(H) = H$.  
\end{proof}

Now we will show that the sequence $\mathbf{x} = \{H(\theta^{n-1}(\mathbf{u}))\}_{n\in \mathbb{Z}}$ is the unique entire solution associated to $\mathbf{u}$. Consequently, by \eqref{eq:robinson}, $\mathbf{A} = \{H(\theta^{-1}(\mathbf{u}))\}_{\mathbf{u} \in \mathcal{U}}$ is the unique pullback attractor of the cocycle $(\theta, \Phi)$. This also means that when training RCs, for any choice of initial condition, the dynamics of the reservoir will become close to the fibres of the pullback attractor after a long washout. The rate of contraction to the pullback attractor is controlled by $\mu$.

\begin{theorem} \label{thm:H_entire_sln}
    Under the assumption of Theorem \ref{thm:the_map_H}, the sequence $\mathbf{x} = \{H(\theta^{n-1}(\mathbf{u}))\}_{n\in \mathbb{Z}}$ is the unique entire solutions associated to $\mathbf{u} \in \mathcal{U}$. 
\end{theorem}

\begin{proof}
    Fix $\mathbf{u} \in \mathcal{U}$ and let $n, m \in \mathbb{Z}, n \ge m$. We note that 
    \begin{align*}
        H(\mathbf{u}) = \Phi(1, \mathbf{u}, H(\theta^{-1}(\mathbf{u})))
    \end{align*}
    by Theorem \ref{thm:the_map_H}. We first show that for any $k \in \mathbb{N}_{0}$, 
\begin{align} \label{eq:H_entire_sln}
    \mathbf{x}[n] = \Phi(k, \theta^{n-k}(\mathbf{u}), \mathbf{x}[n-k]).
\end{align}
    We argue by induction on $k$. The case $k = 0$ follows from the definition of a cocycle. For $k = 1$, 
    \begin{align*}
        \Phi(&1, \theta^{n-1}(\mathbf{u}), \mathbf{x}[n-1]) \\
        &=\Phi(1, \theta^{n-1}(\mathbf{u}), H(\theta^{n-2}(\mathbf{u}))), \\
        &= H(\theta^{n-1}(\mathbf{u})) = \mathbf{x}[n].
    \end{align*}
    Assume (\ref{eq:H_entire_sln}) holds for some $k \ge 0$. Let $\mathbf{z} = \{H(\theta^{n-1}(\mathbf{v}))\}_{n \in \mathbb{Z}}$ be the entire solution associated to $\mathbf{v} = \theta^{-1}(\mathbf{u})$. Then
    \begin{align*}
        \Phi(k+1, \:&\theta^{n-(k+1)}(\mathbf{u}), \mathbf{x}[n-(k+1)]) \\
         &= \Phi(1, \theta^{n-1}(\mathbf{u}), \Phi(k, \theta^{n-(k+1)}(\mathbf{u}), \mathbf{x}[n-(k+1)]) \quad \text{by the cocycle property } (\ref{eq:cocycle_property}),\:  \\
         &= \Phi(1, \theta^{n-1}(\mathbf{u}), \Phi(k, \theta^{n-(k+1)}(\mathbf{u}), H(\theta^{n-1 - (k+1)}(\mathbf{u})), \\
         &= \Phi(1, \theta^{n-1}(\mathbf{u}), \Phi(k, \theta^{n-k}(\mathbf{v}), H(\theta^{n- k-1}(\mathbf{v})))), \\
         &= \Phi(1, \theta^{n-1}(\mathbf{u}),\Phi(k, \theta^{n-k}(\mathbf{v}), \mathbf{z}[n-k])), \\
         &= \Phi(1, \theta^{n}(\mathbf{v}), \mathbf{z}[n])\quad \text{by the inductive hypothesis,} \\
         &= \mathbf{z}[n+1], \\ 
         &= H(\theta^{n}(\mathbf{v})) = H(\theta^{n-1}(\mathbf{u})) = \mathbf{x}[n].
    \end{align*}
    Choosing $k = n - m$, we have that 
    \begin{align*}
        \mathbf{x}[n] = \Phi(n-m, \theta^m(\mathbf{u}), \mathbf{x}[m]).
    \end{align*}    

\end{proof}
\section{Dimension of the pullback attractor} \label{sec:dim_pull}
It is usually insightful to study properties of the skew product flow $(\theta, \Phi)$ in relation to the corresponding property of the driving input sequence or the space of input sequences. For example, it has been recently shown that the echo index of an input-sequence, a generalisation of ESP, can be related to the minimal number of repetitions in the input sequence \cite{Ceni2020echo, ceni2023transitions}. It has also been shown that in training a reservoir computer to reconstruct the Lyapunov spectrum of a dynamical system, the most negative conditional Lyapunov exponent of the reservoir dynamics should be more negative than the most negative Lyapunov exponent of the driving dynamical system \cite{hart2024attractor}. It was shown in \cite{hart2020embedding} that a related condition is sufficient for the GS to be continuously differentiable. RCs have also been shown to display multifunctionality when trained on data from separate attractors \cite{flynn2021multifunctionality, pascanu2013difficulty}. In this section, we link the dimension of the pullback attractor of $(\theta, \Phi)$ to the dimension of the space of input sequences. In particular, we establish upper and lower bounds for the dimension of the pullback attractor of $(\theta, \Phi)$ based on the dimension of the state input sequences $\mathcal{U}$ and its subsets, with a particular focus on subsets of the type $\mathcal{V}_{(U,\phi)}$. 
\\
\\
We will use the box-counting dimension as it provides an upper bound for the other notions of dimension such as Hausdorff dimension \cite{robinson2010embedding}.
\begin{definition}[Box-counting dimension] \label{def:box_dim}
Let $(Z, d)$ be a compact metric space. For any $\varepsilon > 0$, let $N(Z, \varepsilon)$ be the minimum number of balls of radius $\varepsilon$ needed to cover $Z$ with centres in $Z$. Note that $N(Z, \varepsilon)$ is finite for each $\varepsilon > 0$ since $Z$ is compact. The (upper) box-counting dimension is defined as 
\begin{align*}
    \text{\normalfont dim}_B Z = \lim \sup_{\varepsilon \to 0} \frac{\log_2 N(Z, \varepsilon)}{-\log_2 \varepsilon}.
\end{align*}  
\end{definition}

The box-counting dimension satisfies the following properties
\begin{enumerate} 
    \item box-counting dimension of a singleton is $0$, 
    \item the box-counting dimension of an open subset of an $n$-dimensional Euclidean space is $n$. 
    \item if $(X, d_X)$ and $(Z, d_Z)$ are metric spaces and $f:X \to Z$ is a Lipschitz continuous function with 
    \begin{align*}
     d_Z(f(x_2), f(x_1)) \le C d_2(x_2, x_1),
    \end{align*}
    then $\dim_B f(X) \le \dim_B X$. In particular, if $f:X \to Z$ is bi-Lipschitz, then $\dim_B f(X) = \dim_B X$.
\end{enumerate}

\subsection{Upper bound for finite dimensional input sequence spaces} \label{ssec:finite_dim_input_seqs}
There are various results on upper bounds for the dimension of the pullback attractor that can be applied to finite dimensional input sequence spaces \cite{cunha2024smoothing, carvalho2025finite}. Lower bounds, in general, require specific information about the the driving dynamical system such as the existence of a ``physical measure" \cite{young1981capacity} or the existence of saddle-like fixed points \cite{chepyzhov2002attractors}, the dimension of whose unstable manifold provides the lower bound.  For our upper bound, we will use the result of Vishik and Chepyzhov\cite[Chapter IX, Theorem 1.2]{chepyzhov2002attractors}.  Let $\mathcal{V} \subseteq \mathcal{U}$ and 
\begin{align} \label{def:Hcal_V}
    \mathcal{H}_{\mathcal{V}} = \bigcup_{\mathbf{u} \in \mathcal{V}} \mathbf{A}_\mathbf{u} \subseteq \mathbb{R}^{N_r}
\end{align}
be all points on the pullback attractor of the skew product flow $(\theta, \Phi)$ with input sequences in $\mathcal{V}$. We refer to the set $\mathcal{H}_\mathcal{V}$ as the pullback attractor corresponding to input sequences in $\mathcal{V}$. We start with the statement of the result of Vishik and Chepyzhov and then prove that our skew product flow $(\theta, \Phi)$ satisfies its assumptions if $g$ is continously differentiable.

\begin{theorem} [Vishik and Chepyzhov \cite{chepyzhov2002attractors}] \label{thm:vishikchepyzhov}
 Let $\mathcal{V}$ be a finite-dimensional subset of $(\mathcal{U}, d^w)$. Suppose the skew product flow $(\theta, \Phi)$ satisfies the following assumptions. 
 \begin{enumerate}[label=\normalfont \textbf{VC\arabic*}]
    \item \textbf{Finite time Lipschitz continuous} on $\mathcal{V}$ $i.e$ there exist a function $C:\mathbb{N} \to \mathbb{R}_{+}$ such that for all $x \in X$ \label{assume:unif_lipschitz}
    \begin{align*}
        ||\Phi(n, \mathbf{u}, x) -  \Phi(n, \mathbf{v}, x)|| \le C(n) d^w(\mathbf{u}_{[0, n]}, \mathbf{v}_{[0, n]}) 
    \end{align*}
    where
    $\mathbf{u}_{[0, n]}$ is the projection of $\mathbf{u}$ onto $U^{n+1}$ \textit{i.e.} $\mathbf{u}_{[0, n]}[i] = \mathbf{u}[i]$ for $0 \le i \le n$ and $\mathbf{u}_{[0, n]}[i] = 0 \in \mathbb{R}^{N_\text{in}}$ otherwise. 
    \item \textbf{Uniform quasidifferentiability} on $\mathcal{V}$: the skew product flow $(\theta, \Phi)$ is uniformly quasidifferentiable on $\mathcal{V}$ if there exist a parameterised family of bounded linear maps $\{L\Phi(n, \mathbf{u}, x)\}_{\mathbf{u} \in \mathcal{V}, n \in \mathbb{N}_0, x \in X}$ from $\mathbb{R}^{N_r}$ to $\mathbb{R}^{N_r}$ such that \label{assume:unif_q_d}
     \begin{align*}
    ||\Phi(n, \mathbf{u}, x) - \Phi(n, \mathbf{u}, z) - L \Phi(n, \mathbf{u}, x)(x - z)|| \le \gamma(n, x-z) ||x - z|| 
     \end{align*}
    where $\mathbf{u} \in \mathcal{V}$, $x, z \in X$ and $\gamma(n, s) \to 0$ as $||s|| \to 0$ for all $n \in \mathbb{N} \cup \{0\}$. The function $\gamma(n, s)$ is independent of $x, z \in X$ and $\mathbf{u} \in \mathcal{V}$. 
 \end{enumerate}
   Then
\begin{align*}
    \text{ \normalfont dim}_B \: \mathcal{H}_{\mathcal{V}} \le \sup_{\mathbf{u} \in \mathcal{V}} \text{\normalfont dim}_B \: \mathbf{A}_\mathbf{u} + \text{\normalfont dim}_B \: \mathcal{V}.
\end{align*}  
\end{theorem}
If the skew product flow satisfies the echo state property, the following is a simple corollary of the above result and the fact that the box-counting dimension of a singleton is $0$. 
\begin{corollary} \label{cor:dim_BV}
   Suppose $\mathcal{V}$ is a finite dimensional subset of $\mathcal{U}$, and that the skew product flow $(\theta, \Phi)$ satisfies the ESP, then 
   \begin{align*}
       \text{ \normalfont dim}_B \: \mathcal{H}_{\mathcal{V}} \le \text{\normalfont dim}_B \: \mathcal{V}.
   \end{align*}
\end{corollary}
\begin{remark}
    The upper bound in Corollary \ref{cor:dim_BV} would also apply to the case when $\mathbf{A}_\mathbf{u}$ is finite for all $\mathbf{u} \in \mathcal{V}$. Alternatively, the dimension of $\mathbf{A}_\mathbf{u}$ can be bounded above by the singular values of the matrix of the linear map $L\Phi(n, \mathbf{u}, x)$ \cite{chepyzhov2002attractors, temam2012infinite}. In particular, suppose $M(n, \mathbf{u}, x)$ is the matrix of the linear map $L\Phi(n, \mathbf{u}, x)$ and  $\alpha_1(n, \mathbf{u}, x) \ge \alpha_2(\mathbf{u}, n, x) \ldots \ge \alpha_{N_r}(\mathbf{u}, n, x) \ge 0$ are its singular values \textit{i.e.} the square root of the eigenvalues of $M^T M$. Let the numbers $\lambda_j(\mathbf{u})$, $j = 1, \ldots, N_r,$ be defined as     
    \begin{align} \label{eq:uniform_lyapunov_exponent}
        \lambda_j(\mathbf{u}) =  \sup_{x \in X}  \limsup_{n \to \infty} \frac{1}{n} \log \alpha_{j}(\mathbf{u}, n, x).
    \end{align}
    If $m(\mathbf{u}) \in \mathbb{N}_0$ is the smallest integer such that $\sum_{i = 1}^{m(\mathbf{u}) + 1} \lambda_j(\mathbf{u}) < 0$ and $\lambda_2(\mathbf{u}) \ge \lambda_3(\mathbf{u}) \ge \ldots \ge \lambda_{N_r}(\mathbf{u})$, then 
\begin{align*}
        \dim_B \mathbf{A}_{\mathbf{u}} \le m(\mathbf{u}) + \frac{\sum_{i = 1}^{m(u)}  \lambda_j(\mathbf{u})}{- \lambda_{m(\mathbf{u})+1}(\mathbf{u})}.
\end{align*} 
    If the $\limsup$ in \eqref{eq:uniform_lyapunov_exponent} can be replaced with a limit, the numbers 
    \begin{align*}
        \lambda_j(\mathbf{u}, x) = \lim_{n \to \infty} \frac{1}{n} \log \alpha_{j}(\mathbf{u}, n, x), \: j = 1, \ldots, N_r,
    \end{align*}
    are the so-called conditional Lyapunov exponents of $(\theta, \Phi)$ when driven with $\mathbf{u}$ \cite{temam2012infinite, pecora1991driving}. 
\end{remark}
We will now show that assumption \ref{assume:unif_q_d} and \ref{assume:unif_lipschitz} are satisfied by the skew product flow $(\theta, \Phi)$ if $g$ is continuously differentiable on $U \times X$ and there exist $\mu, \eta > 0$ such that

\begin{align} \label{cond:lipschitz}
    ||g(u, x) - g(v, z)|| \le \eta || u - v || + \mu ||x - z||, \: \text{ for all }\: x, y \in X \: \text{and}\: u, v \in U.
\end{align}


\begin{lemma} \label{lmm:vishik_assumptions}
Suppose $g = (g^1, \ldots, g^{N_r})$ is continuously differentiable on $U \times X$ and satisfies condition \eqref{cond:lipschitz}. Then skew product flow $(\theta, \Phi)$ satisfies assumption \ref{assume:unif_lipschitz} and assumption \ref{assume:unif_q_d}. 
\end{lemma}

\begin{proof}
Starting with assumption \ref{assume:unif_lipschitz}, it follows by the cocycle property and assumption \eqref{cond:lipschitz} that
    \begin{align*}
        ||\Phi(n, \mathbf{u}, x) - \Phi(n, \mathbf{v}, x)|| &\le \eta ||\mathbf{u}[n]- \mathbf{v}[n]|| +  \mu ||\Phi(n-1, \mathbf{u}, x) - \Phi(n-1, \mathbf{v}, x)||, \\
        &\le \eta ||\mathbf{u}[n] - \mathbf{v}[n]|| +  \mu\eta ||\mathbf{u}[n-1] - \mathbf{v}[n-1]|| +  \mu^2 ||\Phi(n-2, \mathbf{u}, x) - \Phi(n-2, \mathbf{v}, x)||, \\
        &\quad \vdots \\
        &\le \eta \sum_{i = 0}^{n-1} \mu^{i} ||\mathbf{u}[n-i ] - \mathbf{v}[n-i] || \: \text{ since} \: \Phi(0, \mathbf{u}, x) = \Phi(0, \mathbf{v}, x) = x,\\
        &\le \eta \left ( \sum_{i = 0}^n \frac{\mu^{n-i}}{w_i} \right) \max_{j = 0, \ldots n} w_j || \mathbf{u}[j] - \mathbf{v}[j]||, \\
        &= C(n) d^w(\mathbf{u}_{[0, n]}, \mathbf{v}_{[0, n]}).
    \end{align*}
As for assumption \ref{assume:unif_q_d}, we note that there is nothing to prove for the case $n = 0$ since $\Phi(0, \mathbf{u}, x) = x$. Now, for each $n \in \mathbb{N}$, the map $\Phi(n, \cdot, \cdot): \mathcal{U} \times X \to X$ is equivalent to the map $F: U^n \times X \to X$ where 
    \begin{align*}
        F(  \underline{\mathbf{u}}, x) = g(\mathbf{u}[n], g(\mathbf{u}[n-1], ..., g(\mathbf{u}[1], x), \ldots)) = \Phi(n, \mathbf{u}, x),
    \end{align*}
    $\mathbf{u} \in \mathcal{U}$, $x \in X$ and $\underline{\mathbf{u}} = (\mathbf{u}[1], \ldots, \mathbf{u}[n])$. Since $g$ is continuously differentiable on $U \times X$, $F$ is also continuously differentiable on $U^n \times X$ by the chain rule. We will take the linear map $L\Phi(n, \mathbf{u}, x)$ to be $D_2 F(\underline{\mathbf{u}}, x)$, the partial derivative of $F$ with respect to $x$. The function $\gamma(n, \xi)$ can be defined as 
    \begin{align*}
        \gamma(n, \xi) = \sup_{\underline{\mathbf{u}} \in U^n} \sup_{0 < ||x - z|| \le \xi} \frac{||F(\underline{\mathbf{u}}, x) - F(\underline{\mathbf{u}}, z) -  D_2 F(\underline{\mathbf{u}}, x)(x - z)||}{||x - z||}.
    \end{align*}
    Note that $\gamma$ is well defined since both $U^n$ and $X$ are compact.
    \\
    \\
\end{proof}
\begin{remark}
    If, in addition to $g$ being continuously differentiable, both $X$ and $U$ are convex, then the Lipschitz condition \eqref{cond:lipschitz} holds.
\end{remark}
Theorem \ref{thm:vishikchepyzhov} can be applied to the input sequence space $\mathcal{V}_{(U, \phi)}$ to get an upper bound on the box-counting dimension of $\mathcal{H}_{\mathcal{V}_{(U, \phi)}}$. This result is based on the fact that if $\phi$ is bi-Lipschitz, then there exists a bi-Lipschitz embedding of $U$ into $\mathcal{V}_{(U, \phi)}$.


\begin{theorem} \label{prop:dyn_euclidean_prod}
    Suppose there exist $C > 1$ such that for all $u, v \in U$
    \begin{align*}
        ||\phi(u) - \phi(v)|| \le C||u - v||\: \text{ and } \:||\phi^{-1}(u) - \phi^{-1}(v)|| \le C||u - v||.
    \end{align*}
    Assume $(\theta, \Phi)$ satisfy the ESP and $g$ satisfies the assumptions of Lemma \ref{lmm:vishik_assumptions}, then 
    \begin{align*}
        \dim_B \mathcal{H}_{\mathcal{V}_{(U, \phi)}} \le \dim_B U
    \end{align*}
\end{theorem}

\begin{proof}    
    We will show that there exists a weighting sequence $w = \{w_n\}_{n \in \mathbb{N}_0}$ such that 
   \begin{align*}
       \dim_B \mathcal{V}_{(U, \phi)} = \dim_B U
   \end{align*} 
   and the result will follow by Theorem \ref{thm:vishikchepyzhov}, Corollary \ref{cor:dim_BV} and Lemma \ref{lmm:vishik_assumptions}.
    Let $u \in U$ and define $\Psi:U \to \mathcal{V}_{(U, \phi)}$ such that 
    \begin{align*}
        \Psi(u) = (\ldots, \phi^{-n}(u), \ldots, \phi^{-1}(u), u, \phi(u), \ldots, \phi^n(u), \ldots). 
    \end{align*}
    The map $\Psi$ maps $u$ to its trajectory with $\Psi(U) = \mathcal{V}_{(U, \phi)}$. $\Psi$ is injective and continuous (since each projection is continuous) and thus a homeomorphism onto its image since $U$ is compact. The inverse $\Psi^{-1}:\Psi(U) \to U$ is given by 
    \begin{align*}
        \Psi^{-1}(\mathbf{u}) = \mathbf{u}[0].
    \end{align*}
    We define the weighting sequence $w:\mathbb{N}_0 \to (0, 1)$ as $w_n = C^{-n}$. 
    Then
    \begin{align*}
        d^w(\Psi(u), \Psi(v)) = \sup_{n \in \mathbb{Z}} w_{|n|} d(\phi^n(u), \phi^n(v)) 
        &\le \sup_{n \in \mathbb{Z}} w_{|n|} C^{|n|} d(u, v) = d(u, v).
    \end{align*}
    The inverse $\Psi^{-1}: \Psi(U) \to U$ also satisfies
    \begin{align*}
        d(\Psi(\mathbf{u}), \Psi(\mathbf{v})) = d(\mathbf{u}[0], \mathbf{v}[0]) \le d^w(\mathbf{u}, \mathbf{v})
    \end{align*}
    for any $\mathbf{u}, \mathbf{v} \in \mathcal{V}_{(U, \phi)}$. Thus $\Psi$ is a bi-Lipschitz map from which we conclude that $\dim_B \mathcal{V}_{(U, \phi)}= \dim_B U.$
\end{proof}
\subsection{Lower bound and an infinite-dimensional input sequence space} \label{ssec:input_seq_U_Z}
As mentioned earlier, lower bound for the dimension of attractors, even for autonomous systems, usually require specific information about the dynamics of the system. We will attempt to establish a lower bound for the box-counting dimension of the pullback attractor corresponding to the space of periodic input sequences of sufficiently large period in $\mathcal{U}= U^\mathbb{Z}$ and by extension for $\mathcal{U}$ itself. We prove in Theorem \ref{thm:fractal_dim_H} that if the $U$ interior of $U$ is non-empty $i.e$ $U$ contains an open set, then the pullback attractor corresponding to input sequences of period $m$ will also contain an open set for $m$ sufficiently large. Thus the pullback attractor has box-counting dimension $N_r$. This results is perhaps not surprising for the set $\mathcal{H}_{\mathcal{U}}$ since from \eqref{eq:H_n}, $H(\mathbf{u}) \in \mathcal{H}_{\mathcal{U}}$ depends on infinite past histories of $\mathbf{u}$ from the first term, where $\mathbf{u}$ takes all possible values in an infinite-dimensional metric space $\mathcal{U}$. 
\\
\\
%
    The proof of Theorem \ref{thm:fractal_dim_H} is based on the implicit function theorem \cite[Theorem 17.3]{apostol1974mathematical} and the local surjectivity Theorem \cite{abraham2012manifolds}, which we now state. 
\begin{theorem} [Implicit function theorem] \label{thm:IFT}
    Let $U \subseteq \mathbb{R}^n$ and $X \subseteq \mathbb{R}^m$ be open sets and $G:U \times X \to \mathbb{R}^m$ be a continuously differentiable function on $U \times X \subseteq \mathbb{R}^n \times \mathbb{R}^m$. Suppose $G(u_0, x_0) = 0$ for some  $u_0 \in U$ and $x_0 \in X$, and the Jacobian $D_2 G(u_0, x_0)$ is invertible. Then there exist an open set $V \subseteq \mathbb{R}^n$ containing $u_0$ and a unique continuously differentiable function $h:V \to \mathbb{R}^m$ such that 
    \begin{enumerate}
        \item $h(u_0) = x_0$ and $G(u, h(u)) = 0$ for every $u \in V$,
        \item The Jacobian of $h$ is given by 
        \begin{align*}
            D h(u) = -[D_2 G(u, h(u))]^{-1} [D_1 G(u, h(u))].
        \end{align*}
    \end{enumerate}
\end{theorem}

\begin{theorem} [Local surjectivity theorem]\label{lmm:open_mapping}
    Let $n \ge m$ and $h:\mathbb{R}^n \to \mathbb{R}^m$ be a continuously differentiable function. Suppose there exists $u \in \mathbb{R}^n$ such that $Dh(u)$ is surjective \textit{i.e.} it has rank $m$. Then there exists a neighbourhood $U$ of $u$ and $W$ of $h(u)$ such that $h:U \to W$ is surjective. In particular, if $V$ is an open subset of $\mathbb{R}^m$ and $Dh(u)$ is surjective for all $u \in V$. Then $h$ is an open mapping on $V$. 
\end{theorem}
\begin{remark}
     The local surjectivity theorem is related to the fact that submersions are locally projections and thus are open mappings. It has been recently shown that it holds for an even larger class of functions \cite{saint2016open} so long as its set of critical values and critical points are not too big. 
\end{remark}

Now we state the main theorem of this section. 

\begin{theorem} \label{thm:fractal_dim_H}
    Let $U$, $X$ and $g$ be as in Definition \ref{def:skewprod}. Additionally, we assume that $\normalfont\interior\: U \ne \emptyset$ and $g: U \times X \to X$ is continuously differentiable. Let $m \in \mathbb{N}$ be such that $m N_{\text{in}} \ge N_r$. Let $\underline{\mathbf{u}} = (\underline{\mathbf{u}}_1, \underline{\mathbf{u}}_2, \ldots, \underline{\mathbf{u}}_m) \in U^m \subseteq \mathbb{R}^{m \times N_{\text{in}}}$ and the function $F:U^{m} \times  X \to \mathbb{R}^{N_r}$ defined by
    \begin{align} \label{def:the_map_F}
        F(\underline{\mathbf{u}}, x) = g(\underline{\mathbf{u}}_1, g(\underline{\mathbf{u}}_2, \ldots, g(\underline{\mathbf{u}}_m, x) \ldots )),
    \end{align}
    satisfies the following conditions
    \begin{enumerate}[label=\normalfont \textbf{G\arabic*}]
        \item 
        $||D_2 g(u, x)||_2 < 1$ for all $(u, x) \in U^{} \times  X $. \label{assume:F_contracting}
        \item  
        there exist $z \in X, \underline{\mathbf{v}} \in U^m$ such that $z = F( \underline{\mathbf{v}}, z)$ and the linear map $D_1 F( \underline{\mathbf{v}}, z):\mathbb{R}^{m \times N_{\text{in}}} \to \mathbb{R}^{N_r}$ is surjective $i.e$ it has rank $N_r$. \label{assume:F_surjective}
    \end{enumerate}
    Let $(U^m)^\mathbb{Z}$ be set of all $m$-periodic input sequences of the form 
    \begin{align*}
    ( \ldots, \underline{\mathbf{u}}_m, \underline{\mathbf{u}}_{m-1}, \ldots, \underline{\mathbf{u}}_1, \underline{\mathbf{u}}_m, \underline{\mathbf{u}}_{m-1}, \ldots, \underline{\mathbf{u}}_1, \ldots).
    \end{align*}
    Then $\mathcal{H}_{(U^m)^\mathbb{Z}}$ contains an open set and thus has box-counting dimension $N_r$. 
\end{theorem}

\begin{rmk_main}
    Condition \ref{assume:F_contracting} implies that, for each $\underline{\mathbf{u}} \in U^m$ fixed, $F(\underline{\mathbf{u}}, \cdot)$ is a contraction on $X$. Thus, by the Banach fixed-point theorem, there exists a unique fixed point $z = z(\underline{\mathbf{v}})$ corresponding to $\underline{\mathbf{v}}$, \textit{i.e.} $F(\underline{\mathbf{v}}, z) = z$. Condition \ref{assume:F_surjective} requires that for some $\underline{\mathbf{v}} \in U^m$, we have that the partial derivative of $F$ with respect to $\underline{\mathbf{u}}$ has full rank when evaluated at $(\underline{\mathbf{v}}, z)$. These two conditions are, in general, independent of each other. 
\end{rmk_main}

\begin{proof}[\textit{\textbf{Proof of Theorem  \ref{thm:fractal_dim_H}}}]
    Let $y$ be the unique solution to $F(\underline{\mathbf{w}}, y) = y$ for some fixed $\underline{\mathbf{w}} \in U^m$. From Theorem \ref{thm:H_entire_sln},  $y = H(\theta^{km}(\mathbf{w}))$ for each $k \in \mathbb{Z}$ where $\mathbf{w}$ is the periodic input sequence 
    \begin{align*}
        ( \ldots, \underline{\mathbf{w}}_m, \underline{\mathbf{w}}_{m-1}, \ldots, \underline{\mathbf{w}}_1, \underline{\mathbf{w}}_m, \underline{\mathbf{w}}_{m-1}, \ldots, \underline{\mathbf{w}}_1, \ldots).
    \end{align*}
    So for any subset $\mathbf{V} \subseteq U^m$, the set $\mathcal{H}_\mathcal{V}$ , $\mathcal{V} = \mathbf{V}^\mathbb{Z}$ of all such $y$'s as $\underline{\mathbf{w}}$ takes values in $\mathbf{V}$ is contained in $\mathcal{H}_{(U^m)^\mathbb{Z}}$. In particular, $\text{dim}_B \: \mathcal{H}_{(U^m)^\mathbb{Z}} \ge \text{dim}_B \:  \mathcal{H}_\mathcal{V}$. 
\\
\\
    Now let $G(\underline{\mathbf{u}}, x) = x - F(\underline{\mathbf{u}}, x)$. Then $G$ is continuously differentiable on $ U^m \times X$ with
    \begin{align*}
        D_2 G(\underline{\mathbf{u}}, x) = \mathbb{I} - D_2 g(\underline{\mathbf{u}}_1, y_1)  D_2 g(\underline{\mathbf{u}}_2, y_2)  \ldots  D_2 g(\underline{\mathbf{u}}_{m-1}, y_{m-1})  D_2 g(\underline{\mathbf{u}}_m, y_m)
    \end{align*}
    where 
    \begin{align} \label{def:y_js}
        y_j = g(\underline{\mathbf{u}}_{j+1}, y_{j+1}) \: \text{ for } j < m \: \text{ and } \: y_m = x. 
    \end{align}
    Since 
    \begin{align*}
        ||D_2 g(\underline{\mathbf{u}}_1, y_1) &  D_2 g(\underline{\mathbf{u}}_2, y_2)  \ldots  D_2 g(\underline{\mathbf{u}}_{m-1}, y_{m-1})  D_2 g(\underline{\mathbf{u}}_m, y_m) ||_2  \\
        &\le ||D_2 g(\underline{\mathbf{u}}_1, y_1) ||_2 \cdot || D_2 g(\underline{\mathbf{u}}_2, y_2) ||_2 \cdot  \ldots \cdot || D_2 g(\underline{\mathbf{u}}_{m-1}, y_{m-1}) ||_2 \cdot || D_2 g(\underline{\mathbf{u}}_m, y_m) ||_2 \\
        & < 1,
    \end{align*}
    $D_2 G(\underline{\mathbf{u}}, x)$ is invertible for all $x \in X, \underline{\mathbf{u}} \in U^m$ with 
    \begin{align*}
        [D_2 G(\underline{\mathbf{u}}, x)]^{-1} = \sum_{k = 0}^\infty [ D_2 g(\underline{\mathbf{u}}_1, y_1) &  D_2 g(\underline{\mathbf{u}}_2, y_2)  \ldots  D_2 g(\underline{\mathbf{u}}_{m-1}, y_{m-1})  D_2 g(\underline{\mathbf{u}}_m, y_m) ]^k.
    \end{align*}
    Since $G(\underline{\mathbf{v}}, z) = 0$, by the implicit function theorem (Theorem \ref{thm:IFT}), there exists an open neighbourhood $\mathbf{V}_0$ of $\underline{\mathbf{v}}$ and a unique continuously differentiable function $h:\mathbf{V}_0 \to X$ such that $h(\underline{\mathbf{v}}) = z$ and
    \begin{align*}
        G(\underline{\mathbf{u}}, h(\underline{\mathbf{u}})) = 0 
    \end{align*}
    for all $\underline{\mathbf{u}} \in \mathbf{V}_0$. We note that $h(\mathbf{V}_0) \subseteq \mathcal{H}_{(U^m)^\mathbb{Z}}$. As $D_1 F(z, \underline{\mathbf{v}})$ has rank $N_r$, there is an $N_r \times N_r$ submatrix of $D_1 F(\underline{\mathbf{v}}, z)$ which is invertible. Since $F$ is continuously differentiable, we can find a neighbourhood $X_1$ of $z$ and $\mathbf{V}_1$ of $\underline{\mathbf{v}}$ such that for all $x \in X_1$ and $\underline{\mathbf{u}} \in \mathbf{V}_1$, the corresponding submatrix of $D_1 F(\underline{\mathbf{u}}, x)$ remains invertible. This means $D_1 F(\underline{\mathbf{u}}, x)$ has rank $N_r$ for all $x \in X_1$ and $\underline{\mathbf{u}} \in V_1$. Take $\mathbf{V} = h^{-1}(X_1) \cap \mathbf{V}_0 \cap \mathbf{V}_1$. Hence $\mathbf{V}$ is an open neighbourhood of $\underline{\mathbf{v}}$ satisfying $h(\mathbf{V}) \subseteq X_1$, $D_1 F(h(\underline{\mathbf{u}}), \underline{\mathbf{u}})$ has rank $N_r$ and $G(\underline{\mathbf{u}}, h(\underline{\mathbf{u}})) = 0$ for all $\underline{\mathbf{u}} \in \mathbf{V}$. Thus, the linear map
    \begin{align*}
         D h(\underline{\mathbf{u}}) = -[D_2 G(\underline{\mathbf{u}}, h(\underline{\mathbf{u}}))]^{-1}  [D_1 G(\underline{\mathbf{u}}, h(\underline{\mathbf{u}}))] = [D_2 G(\underline{\mathbf{u}}, h(\underline{\mathbf{u}}))]^{-1}  [D_1 F(\underline{\mathbf{u}}, h(\underline{\mathbf{u}}))],
    \end{align*}
    has rank $N_r$ for all  $\underline{\mathbf{u}} \in \mathbf{V}$. Consequently, by Lemma \ref{lmm:open_mapping}, $h(\mathbf{V}) \subseteq \mathcal{H}_{(U^m)^\mathbb{Z}}$ is open in $\mathbb{R}^{N_r}$ and thus has box-counting dimension $N_r$. We conclude that
    \begin{align*}
        \text{dim}_B \: \mathcal{H}_{(U^m)^\mathbb{Z}} = N_r. 
    \end{align*}
\end{proof}

\subsubsection{Conditions of Theorem \ref{thm:fractal_dim_H} for ESNs} \label{sssec:application_echo}
    Conditions of Theorem \ref{thm:fractal_dim_H} for an ESN (without output feedback) are given by related conditions on the weight matrices $W$ and $W^\text{in}$, and possibly the input $\underline{\mathbf{v}}$ in condition \ref{assume:F_surjective}. While conditions \ref{assume:F_contracting} easily hold when $||W||_2 < 1$, condition \ref{assume:F_surjective} is more delicate. 
    To investigate condition \ref{assume:F_surjective}, we consider the $N_r$-fold iteration map 
    \[
    F(\underline{\mathbf{u}}, x) = \text{tanh} (W^\text{in}\underline{\mathbf{u}}_1 + W\text{tanh} (W^\text{in}\underline{\mathbf{u}}_2 + W\text{tanh} ( \ldots + W\text{tanh} (W^\text{in}\underline{\mathbf{u}}_{N_r} + Wx))).
    \]
    Assuming condition \ref{assume:F_contracting} is already satisfied, let $z \in [-1, 1]^{N_r}$ and  $\underline{\mathbf{v}} = (u, \ldots, u) \in U^{N_r}$ be such that $F(\underline{\mathbf{v}}, z) = z$. Let $D_{1k} F(\underline{\mathbf{v}}, z)$ be the partial derivative of $F$ with respect to the $k^\text{th}$ component of $\underline{\mathbf{v}}$. Then
    \begin{align*}
    D_1 F(\underline{\mathbf{v}}, z) &= \begin{bmatrix}
       D_{11} F(\underline{\mathbf{v}}, z)  & | & 
       D_{12} F(\underline{\mathbf{v}}, z)  & | & 
       \ldots  & | & 
       D_{1{N_r}} F(\underline{\mathbf{v}}, z)
    \end{bmatrix} \\
    &= \begin{bmatrix}
       S_1W^\text{in}  & | & 
       S_2    W^\text{in}  & | & 
       \ldots  & | & 
       S_{N_r}    W^\text{in}  
    \end{bmatrix} 
    \end{align*}
    where  
    \begin{align*}
        S_j &= \begin{cases}
            \diag (\sech^2 (W^\text{in}u + Wz)) &\quad\text{if} \: j = 1, \\
            \prod_{i = 1}^{j-1} \diag (\sech^2 (W^\text{in}u + Wz)) {W}&\quad\text{otherwise}.
        \end{cases}
    \end{align*}
    Taking 
    \[
    D = \diag (\sech^2 (W^\text{in}u + Wz)),
    \]
    we have that
    \begin{align*}
        D_1 F(\underline{\mathbf{v}}, z)  &= \begin{bmatrix}
       S_1W^\text{in}  & | & 
       S_2    W^\text{in}  & | & 
       \ldots  & | & 
       S_{N_r}    W^\text{in}  
    \end{bmatrix} \\
    &= \begin{bmatrix}
       DW^\text{in}  & | & 
       DW   W^\text{in}  & | & 
       \ldots  & | & 
       (DW)^{N_r-1}   W^\text{in}  
    \end{bmatrix}.
    \end{align*}
    To show that $D_1 F(\underline{\mathbf{v}}, z)$ is full rank, it is sufficient to show that it has an $N_r \times N_r$ submatrix which is full rank. If we take $C \in \mathbb{R}^{N_r \times 1}$ to be the first (or any) column of $W^\text{in}$, $D_1 F(\underline{\mathbf{v}}, z)$ has full rank if the $N_r \times N_r$ matrix 
    \begin{align} \label{eq:resembles_kalman}
\begin{bmatrix}
       DC  & | & 
       DW   C  & | & 
       \ldots  & | & 
       (DW)^{N_r-1}   C
    \end{bmatrix} \quad \text{has full rank.}
    \end{align} 
    Since $W$ and $W^\text{in}$ are randomly generated when training an ESN, a natural approach to \eqref{eq:resembles_kalman} is to show that $W$ and $W^\text{in}$ would satisfy it with probability $1$ with respect to the probability measure that is used to generate $W$ and $W^\text{in}$. This approach was used in \cite{hart2021strange} to show that when entries of $W$ and $W^\text{in}$ are drawn from a regular probability distribution\footnote{A probability distribution is said to be regular if the probability of observing an event with one element is $0$.}, the conditions for the echo state embedding theorem for linear reservoir systems hold with probability $1$. However, the nonlinear dependence of the diagonal matrix $D$ on $W$, $W^\text{in}$, and the fixed point $z$ makes it much harder to generalise this measure-theoretic approach to our setting. Thus, we have opted for a topological approach to \eqref{eq:resembles_kalman}. In particular, we will show that the weight matrices $(W, W^\text{in})$ with $||W||_2 < 1$ satisfying \eqref{eq:resembles_kalman} form a generic (open and dense) set in $\mathbf{\Lambda} \times \mathbb{R}^{N_r \times N_\text{in}}$ where
    \begin{align*}
        \mathbf{\Lambda} = \{W \in \mathbb{R}^{N_r \times N_r}: ||W||_2 < 1\}.
    \end{align*}
    This topological genericity is, in general, weaker than having \eqref{eq:resembles_kalman} holds with probability $1$. This is because a generic set is not necessarily a full measure set for any well-defined measure. For example, consider the rational numbers $\{q_1, q_2, ....,q_n, \ldots \}$ in $[0, 1]$. The union, $Q_k$, of open balls of radius $2^{-n-k},\: n, k \in \mathbb{N}$ centred at $q_n$ is open and dense in $[0, 1]$ and has (Lebesgue) measure at most $2^{-k}$. The set $\bigcap_{k \in \mathbb{N}} Q_{k}$, which now has measure $0$, is open and dense by the Category theorem of Baire \cite{munkres2000topology}. More counterexamples can be found in \cite{hunt1992prevalence}. 
    \\
    \\
    We start with the case where $U$ contain $0 \in \mathbb{R}^{N_\text{in}}$. In this case, the fixed point is $z = 0 \in \mathbb{R}^{N_r}$ for $\underline{\mathbf{v}} = (0, \ldots, 0) \in U^{N_r}$. Thus $D$ is the identity and the matrix in \eqref{eq:resembles_kalman} gives the Kalman controllability matrix \cite{kalman1970lectures, sontag2013mathematical}
     \begin{align} \label{eq:kalman}
\begin{bmatrix}
       C  & | & 
       W   C  & | & 
       \ldots  & | & 
       W^{N_r-1}   C
    \end{bmatrix},
    \end{align} 
which is known to be full rank for a generic choice of $W$ and $C$, and almost surely if the entries of $W$ and $C$ are sampled from a regular probability distribution \cite[Proposition 4.4]{hart2021strange}. 
\\
\\
Now suppose $0 \not \in U$. We start by observing that the genericity of the Kalman controllability matrix being full rank also means the set of triples $(W, C, D)$ for which \eqref{eq:resembles_kalman} holds is generic\footnote{The determinant of a matrix is polynomial function of its entries. It is known that if a polynomial function is non-vanishing, its roots form a closed nowhere subset of $\mathbb{R}$.} in $\mathbb{R}^{N_r \times N_r} \times \mathbb{R}^{N_r} \times \mathbb{R}^{N_rN_\text{in}}$. Thus for any choice of $(W, W^\text{in}) \in \mathbf{\Lambda} \times \mathbb{R}^{N_r \times N_\text{in}}$ and $u\in U$, we can find matrices ${W}_0$, ${W}_0^\text{in}$ and a diagonal matrix ${D}_0$ arbitrarily close to $W$, $W^\text{in}$ and $\normalfont \diag (\sech^2 (W^\text{in}u + Wz))$ respectively such that 
    \begin{align*}
\begin{bmatrix}
       {D}_0{W}_0^\text{in}  & | & 
       {D}_0{W}_0   {W}_0^\text{in}  & | & 
       \ldots  & | & 
       ({D}_0{W}_0)^{N_r-1}  {W}_0^\text{in}  
    \end{bmatrix}
\end{align*}
    has full rank. The problem is their no guarantee that we can find $\underline{\mathbf{w}} = (\tilde{u}, \ldots, \tilde{u}) \in U^{N_r}$ and $\tilde{z} \in [-1, 1]^{N_r}$ such that $F(\underline{\mathbf{w}}, \tilde{z}) = \tilde{z}$ and $\normalfont \diag (\sech^2 ({W}_0^\text{in}\tilde{u} + {W}_0\tilde{z}))$ corresponds to any such ${{D}_0}$. Even if we pick an open set $\mathbf{D}$ around ${D}_0$ such that for all $D \in \mathbf{D}$,
\begin{align} \label{eq:Kalman_nearby}
\begin{bmatrix}
       D{W}_0^\text{in}  & | & 
       D{W}_0   {W}_0^\text{in}  & | & 
       \ldots  & | & 
       (D{W}_0)^{N_r-1}   {W}_0^\text{in}
    \end{bmatrix}
\end{align}
remains full rank, it is still not obvious we can find an input $\tilde{u} \in U$ for which $\normalfont \diag (\sech^2 ({W}_0^\text{in}\tilde{u} + {W}_0\tilde{z})) \in \mathbf{D}$. The remainder of this section is to show that in any neighbourhood of $W$ and $W^\text{in}$, a choice of $W_0$, $W_0^\text{in}$ and $\tilde{u}$ corresponding to $D_0$ can be found.
\\
\\
 We will be assuming that $W^\text{in}u \ne 0$. For such choices of $W^\text{in}$ and for any $W \in \mathbf{\Lambda}$, the fixed point $z$ is non-zero. Additionally, the set of all such ${W}^\text{in} \in \mathbb{R}^{N_r \times N_\text{in}}$ is open and dense in $\mathbb{R}^{N_r \times N_\text{in}}$. Openness follows from the fact that the linear map $W^\text{in} \to W^\text{in}u$ is continuous while density follows from the that if $u_i \ne 0$ and 
 \begin{align*}
    \sum_{j = 1}^{N_\text{in}} W_{kj}^\text{in}u_j = 0,
 \end{align*}
 then the matrix $W_0^\text{in}$, which has the same entries as $W^\text{in}$ except that $W_{0ki}^\text{in} = W_{ki}^\text{in} + w$ for some $w \ne 0$, satisfies $W_0^\text{in} u \ne 0$. 
\\
\\
Considering $\mathbf{\Lambda} \times \mathbb{R}^{N_r \times N_\text{in}}$ in the subspace topology of $\mathbb{R}^{N_r^2 + N_rN_\text{in}}$ endowed with the norm $||\cdot||$, let
\begin{align*}
    \mathbf{M} := \{(W, W^\text{in}) \in \mathbf{\Lambda} \times \mathbb{R}^{N_r \times N_\text{in}}: \text{ such that }\eqref{eq:resembles_kalman} \: \text{holds} \}.
\end{align*}
The following theorem proves that $\mathbf{M}$ is dense in $\mathbf{\Lambda} \times \mathbb{R}^{N_r \times N_\text{in}}$. 
\begin{theorem} \label{lmm:open_echo_state}
    Let $W \in  \mathbf{\Lambda}$ and $W^\text{in} \in \mathbb{R}^{N_r \times N_\text{in}}$ be two full rank matrices. Suppose $u \in U$ is non-zero such that the unique fixed point $z = \tanh(Wz + W^\text{in} u)$ is also non-zero. Let $\varepsilon > 0$ be given. There exist $(W_0, W_0^\text{in}) \in \mathbf{\Lambda} \times \mathbb{R}^{N_r \times N_\text{in}}$ such that
   \begin{align*}
       ||(W_0, W_0^\text{in}) - (W, W^\text{in})|| \le \varepsilon, \\
       z_0 =  \tanh(W_0z_0 + W_0^\text{in} u)
   \end{align*} 
    and for $D = \normalfont \diag (\sech^2 (W_0z_0 + W_0^\text{in} u))$,
    \begin{align*}
    \begin{bmatrix}
       DW_0^\text{in}  & | & 
       DW_0   W_0^\text{in}  & | & 
       \ldots  & | & 
       (DW_0)^{N_r-1}   W_0^\text{in}  
    \end{bmatrix}
    \end{align*}
 has full rank.
\end{theorem}

\begin{proof}
   Let $G:\mathbf{\Lambda} \times \mathbb{R}^{N_r \times N_\text{in}} \times \mathbb{R}^{N_r}  \to \mathbb{R}^{N_r}$ be defined as 
   \begin{align*}
       G(\Omega, x) = x-\tanh(\widetilde{W}x + \widetilde{W}^\text{in} u), \: \Omega = (\widetilde{W}, \widetilde{W}^\text{in}).
   \end{align*}
   So, $G(\Omega_0, z)=0$ where $\Omega_0 = (W, W^\text{in})$. Since $||W||_2 < 1$, $D_2 G(\Omega_0, z)$ is invertible. Thus by the implicit function theorem, we can find an open neighbourhood $\mathbf{\Omega}_0$ of $\Omega_0$ and a unique continuously differentiable function $h:\mathbf{\Omega}_0 \to \mathbb{R}^{N_r}$ satisfying $G(\Omega, h(\Omega)) = 0$ for all $\Omega \in \mathbf{\Omega}_0$. Now we will show that $D_1 G(\Omega_0, z)$ is surjective and thus $h$ is locally an open mapping by Lemma \ref{lmm:open_mapping}. For $(w_1, w_2) \in \mathbb{R}^{M}$, $K = N_r^2 + N_r N_\text{in}$, we have that
\begin{align*}
   \underbrace{D_1 G(\Omega_0, z)}_{\mathbb{R}^{N_r \times K}} \cdot (w_1, w_2) = -\underbrace{\diag {(\sech^2(Wz + W^\text{in}u))}}_{\mathbb{R}^{N_r \times N_r}}(\underbrace{w_1z + w_2u}_{\mathbb{R}^{N_r \times K}})
\end{align*}
by the chain rule \cite{abraham2012manifolds}. Let $R((w_1, w_2)) = w_1z + w_2u$. Writing $w_1 = (w_{11}, w_{12}, \ldots, w_{1N_r^2})$, and $w_2 = (w_{21}, w_{22}, \ldots, w_{2N_rN_{\text{in}}})$, $R$ can be written as 
\begin{align*}
    R((w_1, w_2)) = \begin{bmatrix}
        R_z & | & R_u
    \end{bmatrix} \begin{bmatrix}
        w_1^T \\ w_2^T 
    \end{bmatrix}
\end{align*}
where $R_z$ and $R_u$ are the matrix of the linear maps $w_1 \to w_1z$ and $w_2 \to w_2u$. In particular, the matrix $R_z \in \mathbb{R}^{N_r \times N_r^2}$ is given by 
\begin{align*}
R_z = \begin{bmatrix}
z^T& 0 & \cdots &  0 \\
0 & z^T &\cdots & 0 \\
0 & 0 & \ddots &\vdots\\
\vdots &\vdots&\cdots & z^T 
\end{bmatrix}
\end{align*}
 which has rank $N_r$ since $z \ne 0$. Thus the derivative of $h$ at $\Omega_0$, given by
 \begin{align}
     D h(\Omega_0) = -[D_2 G(\Omega_0, z)]^{-1}[D_1 G(\Omega_0, z)],
 \end{align}
 has rank $N_r$. It follows that  $h$ is an open mapping on a sufficiently small neighbourhood $\mathbf{\Omega}$ of $\Omega_0$. The neighbourhood $\mathbf{\Omega}$ can be chosen to be inside the ball $B_\varepsilon(\Omega_0)$.
\\
\\
 The existence of the open mapping $h$ on $\mathbf{\Omega}$ means we can find $\Omega_1 = ({W}_1, {W}_1^\text{in})$ inside $B_\varepsilon(\Omega_0)$ such that none of the components of $h(\Omega_1)$ is $0$ \textit{i.e.} $(h(\Omega_1))_i \ne 0$ for all $i = 1, \ldots, N_r$. This also means none of the components of $W_1h(\Omega_1) + {W}_1^\text{in}u$ is $0$ since $h(\Omega_1) = \tanh{(W_1h(\Omega_1) + {W}_1^\text{in}u)}$. Now consider the function $P: \mathbf{\Omega} \to \mathbb{R}^{N_r}$ given by 
  \begin{align} \label{eq:P}
     P(\Omega) = \sech^2(\widetilde{W}h(\Omega) + \widetilde{W}^\text{in}u). 
 \end{align}
 $P$ is continuously differentiable on $\mathbf{\Omega}$ with 
\begin{align*}
    DP(\Omega_1)\cdot (w_1, w_2) &= \diag  (\tau_1(W_1 h(\Omega_1) + {W}_1^\text{in}u)) D(\tanh^{-1}(h(\Omega_1))\cdot (w_1, w_2) \\
    &= \diag  (\tau_1(W_1 h(\Omega_1) + {W}_1^\text{in}u)) \diag(\tau_2(h(\Omega_1)) )Dh(\Omega_1)\cdot (w_1, w_2) 
\end{align*} 
where 
\begin{align*}
    \tau_1(s) = \frac{d }{ds}\sech^2(s) = -2\sech^2(s) \tanh(s), \: s \in \mathbb{R} \: \text{ and }\: \tau_2(t) =\frac{d }{dt}\tanh^{-1}(t) = \frac{1}{1-t^2}, \: -1 < t < 1.
\end{align*}
Since
\[
\tau_1(s) = 0 \iff s = 0 
\]
and  none of the entries of $W_1 h(\Omega_1) + {W}_1^\text{in}u$ is $0$,
the diagonal matrix $ \diag  (\tau_1(W_1 h(\Omega_1) + {W}_1^\text{in}u))$ is invertible. 
Hence $DP(\Omega_1)$ has rank $N_r$ and $P$ is also an open mapping on a sufficiently small neighbourhood $\mathbf{\Omega}_1$ of $\Omega_1$. Again, we can also choose $\mathbf{\Omega}_1$ to be inside the ball $B_\varepsilon(\Omega_0)$. 
\\
\\
Thus the set  
\begin{align*}
    \mathbf{M}_{\mathbf{\Omega}_1} := \{(W, W^\text{in}, D) \in \mathbf{\Omega}_1 \times P(\mathbf{\Omega}_1) : D = P(W, W^\text{in})\}. 
\end{align*}
is an open and non-empty subset of $\mathbf{\Lambda} \times \mathbb{R}^{N_r \times N_\text{in}} \times \mathbb{R}^{N_r}$. Since the set $\mathbf{N}$ of triples $(W, W^\text{in}, D) \in \mathbb{R}^{N_r^2 + N_rN_\text{in} + N_r}$ for which 
\begin{align*}
\begin{bmatrix}
       DW^\text{in}  & | & 
       DW   W^\text{in}  & | & 
       \ldots  & | & 
       (DW)^{N_r-1}   W^\text{in}  
    \end{bmatrix} \quad \text{has full rank},
\end{align*}
is open and dense in $\mathbb{R}^{N_r^2 + N_rN_\text{in} + N_r}$, $\mathbf{M}_{\mathbf{\Omega}_1} \cap \mathbf{N}$ is open and dense in $\mathbf{M}_{\mathbf{\Omega}_1}$. Consequently, we can choose $(W_0, W_0^\text{in})$ in $B_\varepsilon(\Omega_0)$ such that for $D_0 = \diag P(W_0, W_0^\text{in})$,
    \begin{align*}
    \begin{bmatrix}
       D_0W_0^\text{in}  & | & 
       D_0W_0   W_0^\text{in}  & | & 
       \ldots  & | & 
       (D_0W_0)^{N_r-1}   W_0^\text{in}  
    \end{bmatrix}
    \end{align*}
has full rank. 

\end{proof}
Now we show openness. The proof of Theorem \ref{lmm:open_echo_state} actually establishes that $\mathbf{M}$ is open in $\mathbf{\Lambda} \times \mathbb{R}^{N_r \times N_\text{in}}$, but we will state and prove it here for completeness. 

\begin{theorem} 
    The set $\mathbf{M}$ is an open subset of $\mathbf{\Lambda} \times \mathbb{R}^{N_r \times N_\text{in}}$.
\end{theorem}
\begin{proof}
The proof of Theorem \ref{lmm:open_echo_state} also establishes that, if $||W|| < 1$, the function $P$ \eqref{eq:P} is locally continuously differentiable. Thus, the diagonal matrix $D(W, W^\text{in}) = \diag P(W, W^\text{in}) = \diag{(\sech^2(Wh(W, W^\text{in}) + W^\text{in}u))}$ depends continuously on the entries of $W$ and $W^\text{in}$. Let $C$ be the first column of $W^\text{in}$ and $\det Q(W, W^\text{in})$ be the determinant of the matrix 
\begin{align*}
 Q = \begin{bmatrix}
       D C  & | & 
       DW C  & | & 
       \ldots  & | & 
       (DW)^{N_r-1}   C  
    \end{bmatrix}.
\end{align*}
Since $\det Q(W, W^\text{in})$ is a polynomial function of the entries of $C$, $W$ and $D$, it is continuous in the entries of $W$ and $W^\text{in}$. Thus $\det Q^{-1} (\mathbb{R} \backslash \{0\}) \cap (\mathbf{\Lambda} \times \mathbb{R}^{N_r \times N_\text{in}}) = \mathbf{M}$ is open.
\end{proof}

\section{Numerical Experiments} \label{sec:numerical}
We can numerically verify the upper bound in Theorem \ref{thm:vishikchepyzhov} and the lower bound in Theorem \ref{thm:fractal_dim_H} by generating ESNs of moderate size, driving them with input sequences satisfying the conditions of each theorem. For the upper bound, the Lorenz and Rössler system were used to drive the ESNs and the box-counting dimension was estimated for the set containing all states after a very long washout. The long washout here is chosen to be about $95\%$ of the timesteps. For the lower bound, $N_\text{in}$-dimensional periodic input sequences of varying periods $m$ were used to drive another batch of ESNs. In this case, points on the pullback attractor corresponds to fixed points of $m$-fold iteration map and its iterates upto $m-1$. Since the fixed point problem can be solved reasonably fast, we increasingly sample more input sequences as the dimension of the ESN increases. This is mainly due to fact that if a set extended by $1$ dimension, one would need exponentially more points to cover it with balls of the same radius. 
\\
\\
To verify the lower bound, ESNs with $N_r = 5, 10$ and $15$ hidden units were driven with $50$, $100$ and $10000$ random initial periodic input sequences of dimension $N_\text{in} = 2, 5$ and $10$ and periods $m = 5, 10$ and $15$, respectively. The periodic input sequences take their values in $[0, 1]$. The entries of the weight matrices $W$ and $W^\text{in}$ of each ESN was randomly sampled between $-1$ and $1$ and then $W$ rescaled so that $||W||_2 = 0.99$. Both the fixed points and their iterates up to $m-1$ were included in calculating the box-counting dimension. The box-counting dimension is then approximated using the algorithm of Grassberger and Procaccia \cite{grassberger1983measuring} implemented in DynamicalSystems.jl \cite{datseris2022nonlinear}\footnote{The algorithm of \cite{grassberger1983measuring} actually computes the $q$-order Renyi dimension where $q = 2$. The box-counting dimensions corresponds to the $0$-order Renyi dimension and is the largest of all $q$-order Renyi dimensions for $q \ge 0$. We have opted to compute the $2$-order Renyi dimension using \cite{grassberger1983measuring} since it is more accurate than algorithms for computing the box-counting dimension. See \cite{datseris2022nonlinear} for more details.}. The result obtained are summarised in the Table \ref{table:fractal} below. 

\begin{table}[H] \label{table:fractal}
\begin{minipage}{4.1cm}
\centering
\begin{tabular}{ |p{3cm}||p{3cm}|p{3cm}|p{4cm}|  }
 \hline
$N_r$ & $N_\text{in}$ & $m$ & $\dim_B$  \\
 \hline
$5$ & $2$ & $30$ & $4.368$\\
 $10$&   $5$  &$60$ & $9.075$\\
 $15$ & $10$ &$100$ & $13.327$ \\
 \hline
\end{tabular}
\end{minipage}
   \caption{Box-counting dimension estimates for the set of fixed points and their $m-1$ iterates of the $m$-fold iteration map of an ESN with $N_r$ units for inputs in $U^m$, $U \subseteq \mathbb{R}^{N_\text{in}}$.} 
\end{table}

 For the upper bound, we drive echo state networks with $N_r = 10, 20$ and $30$ hidden units with input sequence coming from trajectories the Lorenz system \cite{lorenz1963deterministic}
\begin{align*}
    \dot{u}_1 &= \sigma(u_2 - u_1)\\
    \dot{u}_2 &= u_1(\rho - u_3) - u_2\\
    \dot{u}_3 &= u_1u_2 - \beta u_3,
\end{align*}
and the Rössler system \cite{rossler1976equation}
\begin{align*}
    \dot{u}_1 &= -u_2 - u_3\\
    \dot{u}_2 &= u_1 + au_2\\
    \dot{u}_3 &= b + u_3(u_1 - c)
\end{align*}
for classical parameters $\rho = 28, \sigma = 10, \beta = 8/3$, and  $a = b = 0.2$ and $c = 5.7$. We randomly sample $20$ initial conditions close to each attractor and then numerically time-integrated them with a step size of $0.01$ from $t = 0$ to $t = 50$. The first $t = 10$ units of time were discarded as transient. The trajectories then drive the ESNs whose weight matrices $W$ and $W^\text{in}$ were generated from a uniform distribution between $-1$ and $1$ with the matrix $W$ scaled such that $||W||_2 = \rho$, $\rho \in (0, 1.5)$. The ESNs are initialised at the origin and then driven with both the full systems (from each attractor) and projection $u_1$ to the first component, scaled by $0.01$. This scaling ensures that ESN states does not saturate at the tails. All of the ESNs are let to run for $4000$ time-steps. The final $200$ time-steps were recorded in an $4000 \times N_r$ matrix. Figure \ref{fig:lorenz_rossler} shows the estimated box-counting dimension of the resulting trajectories from each of the systems in orange (Lorenz attractor: $2.08$ and Rössler attractor: $1.85$) and the corresponding ESNs. The upper bound fail for most values of $\rho > 1$, which could be due to lost of ESP and possibly the emergence of pullback attractors with continuous fibres. Remarkably, the ESN with $N_r = 10$ seems to have box-counting dimension below that of the attractor for most values of $\rho \in (0, 1.5)$. The dash black line is the inverse of the absolute value of the most negative Lyapunov exponent of each system. Estimates of the box-counting dimension below this tend to be closer to estimate from the trajectory of the system, while estimates above this tend to be smaller than the systems for $\rho < 1$. This maybe related to phenomenon observed in \cite{hart2024attractor}, where ESNs are unable to reconstruct a chaotic attractor if they contract slower than the inverse of the absolute value of the most negative Lyapunov exponent or in \cite{hart2020embedding} where an ESN with spectral radius (which is smaller than $\rho$) equal $1$ was able to reconstruct all but the negative Lyapunov exponent of the Lorenz system. Similar behaviour is observed in \cite{pathak2017using}. 


 \begin{figure} [H]  \label{fig:lorenz_rossler}
    \centering
    \includegraphics[width=0.6\textwidth, height=0.6\textwidth, keepaspectratio]{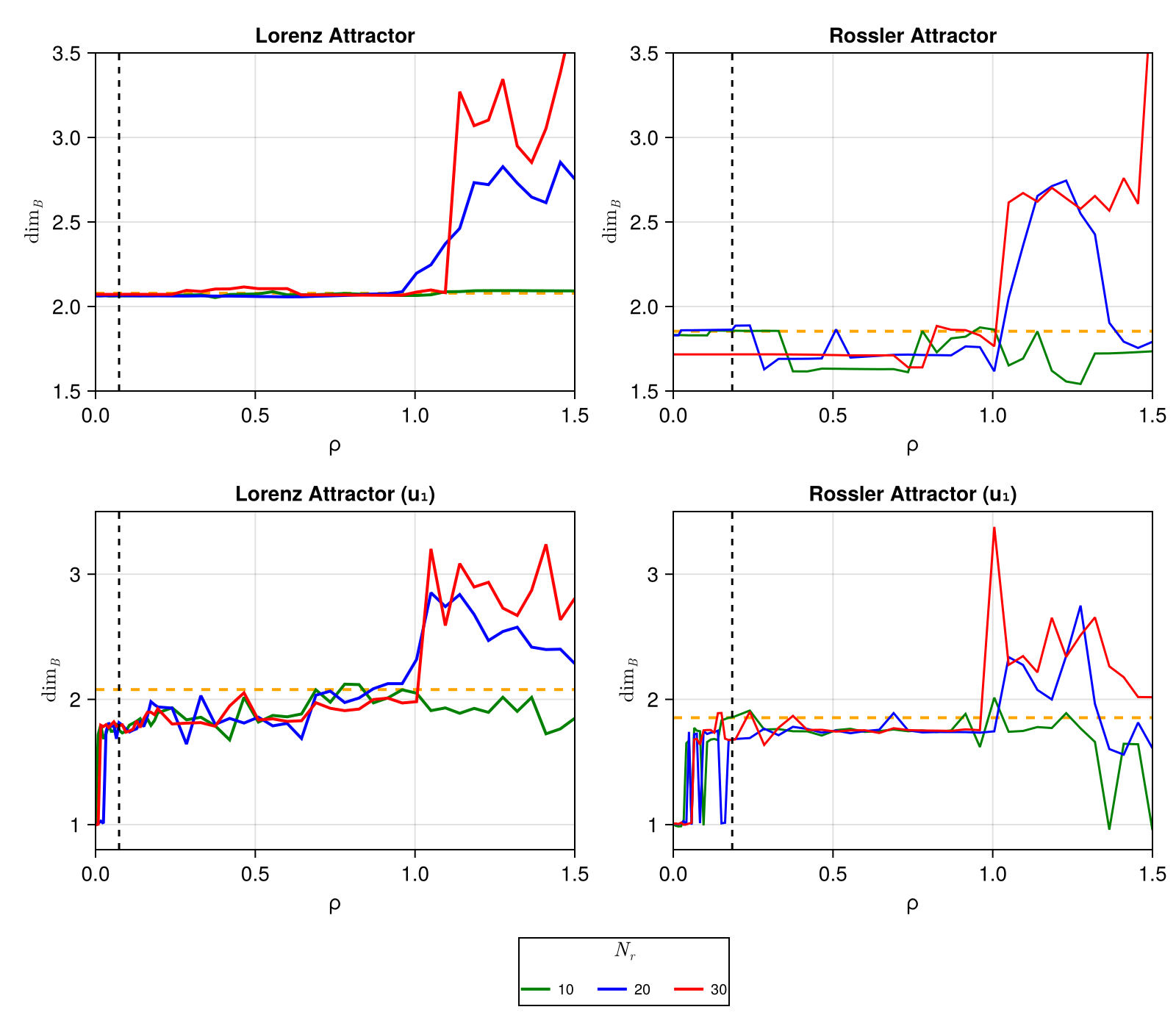}
    \caption{Box-counting dimension of the set of states from ESNs with $N_r = 10, 20$ and $30$ hidden units driving with the dynamics of the Lorenz (left) and Rössler attractors (right), as well as their first components $u_1$. The orange dash line is the estimate box-counting dimension for the trajectories of each system that was used to drive the ESNs. The dash black vertical line is the rate of contraction $\rho$ is equal to the inverse of the absolute value of the most negative Lyapunov exponents of each system.}
\end{figure}


 \section{Discussions and conclusions}  \label{sec:discussions}
 In this paper, we provide two important results for the box-counting dimension of the pullback attractor of reservoir systems. First, using a result of Vishik and Chepyzhov \cite{chepyzhov2002attractors}, we demonstrate that the box-counting dimension of the pullback attractor of a driven reservoir system is bounded above by the box-counting dimension of the space of input sequences in the product topology. We use this result to show that, in particular, the pullback attractor corresponding to input sequences coming from a trajectory of bi-Lipschitz homeomorphism $\phi$ on a compact set $U \subseteq \mathbb{R}^{N_\text{in}}$ is bounded above by $\dim_B U$. It would be interesting to establish a lower bound for the box-counting dimension of the pullback attractor in this case. The main preliminary result used to prove lower bounds on the dimension of sets is Frostman Lemma \cite{mattila1999geometry} and its variants. The Frostman Lemma is based on a well-defined probability (or $\sigma$-finite) measure. However, it is unclear how one can define a measure on the pullback attractor. As is often the case, we can work in the extended state space $\mathcal{V}_{(U, \phi)} \times X$ to get an autonomous dynamical system. Even then, an invariant measure on the input dynamics may be required for any well-defined product measure on the product space. Our experiments with the Lorenz and Rössler attractors suggest that, if $\mu$ is the invariant measure on each of these attractors, the invariant measure on the extended space of the input and the reservoir is just a product of $\mu$ with itself.
 \\
 \\
 The second main result of this paper is that, if $U$ contains an open set, the pullback attractor corresponding to input sequences in $U^\mathbb{Z}$ always contains an open set, and so has dimension equal to that of the driven reservoir. The proof technique relied on a subset of input sequences composed of periodic input sequences whose period is larger than the dimension of the reservoir. This essentially reduces the problem to a fixed-point problem, where tools such as the implicit function theorem and the local surjectivity theorem can be applied. 
\\
\\
A significant amount of work has already been done in training reservoirs on dynamical systems and utilising autonomous dynamics in the prediction phase to infer properties of the dynamical system, such as fractal dimension and  Lyapunov exponents \cite{pathak2017using}, bifurcation and tipping points \cite{patel2023using, kong2021machine}, persistent homology \cite{hart2020embedding} etc. The accuracy of these methods and how it relates to the dimension or number of neurons in the reservoir, as well as the hyperparameters of the network, such as the spectral radius, is yet to be studied. Although nonautonomous dynamical systems theory offers a way to link the dynamics of the input to those of the reservoir, it is not as fully developed as the theory of autonomous dynamical systems. Thus, tackling these questions in generality is a very challenging task. This reason is part of the motivation for the research conducted in this paper. Results of Section \ref{sec:numerical} suggest that for specific quantities, such as the box-counting dimension, it is not necessary to train the reservoir to obtain reasonable estimates. Indeed, if generalised synchronisation is an embedding, this can be done for any dynamics invariant. One of the main consequences of computing these quantities before training is that we can study the regime where their estimates have very high precision, and thus use this as a means to inform how hyperparameters, such as the spectral radius, should be chosen before training. 

\section*{Acknowledgements}
  The author thanks Peter Ashwin and George Datseris for their insightful comments. The author would also like to thank an anonymous reviewer whose suggestions greatly enhance the material in this paper. This work originated from the author's PhD thesis, funded by EPSRC studentship funding via EP/W523859/1. The author is currently supported by an ARIA grant through SCOP-PR01-P003, under the AdvanTip project. For open access purposes, the author has applied a Creative Commons Attribution (CC BY) licence to any Author Accepted Manuscript version resulting from this submission.

\bibliographystyle{plain}
\bibliography{references}

@article{ceni10,
    title = {Interpreting Recurrent Neural Networks Behaviour via Excitable Network Attractors},
    year = {2020},
    journal = {Cognitive Computation},
    author = {Ceni, Andrea and Ashwin, Peter and Livi, Lorenzo},
    number = {2},
    month = {3},
    pages = {330--356},
    volume = {12},
    publisher = {Springer},
    doi = {10.1007/s12559-019-09634-2},
    issn = {18669964},
    arxivId = {1807.10478}
}

@article{panahi2024machine,
  title={Machine learning prediction of tipping in complex dynamical systems},
  author={Panahi, Shirin and Kong, Ling-Wei and Moradi, Mohammadamin and Zhai, Zheng-Meng and Glaz, Bryan and Haile, Mulugeta and Lai, Ying-Cheng},
  journal={Physical Review Research},
  volume={6},
  number={4},
  pages={043194},
  year={2024},
  publisher={APS}
}

@book{sontag2013mathematical,
  title={Mathematical control theory: deterministic finite dimensional systems},
  author={Sontag, Eduardo D},
  volume={6},
  year={2013},
  publisher={Springer Science \& Business Media}
}

@article{pathak2017using,
  title={Using machine learning to replicate chaotic attractors and calculate Lyapunov exponents from data},
  author={Pathak, Jaideep and Lu, Zhixin and Hunt, Brian R and Girvan, Michelle and Ott, Edward},
  journal={Chaos: An Interdisciplinary Journal of Nonlinear Science},
  volume={27},
  number={12},
  year={2017},
  publisher={AIP Publishing}
}

@article{rossler1976equation,
  title={An equation for continuous chaos},
  author={R{\"o}ssler, Otto E},
  journal={Physics Letters A},
  volume={57},
  number={5},
  pages={397--398},
  year={1976},
  publisher={Elsevier}
}

@article{patel2023using,
  title={Using machine learning to anticipate tipping points and extrapolate to post-tipping dynamics of non-stationary dynamical systems},
  author={Patel, Dhruvit and Ott, Edward},
  journal={Chaos: An Interdisciplinary Journal of Nonlinear Science},
  volume={33},
  number={2},
  year={2023},
  publisher={AIP Publishing}
}

@article{kong2021machine,
  title={Machine learning prediction of critical transition and system collapse},
  author={Kong, Ling-Wei and Fan, Hua-Wei and Grebogi, Celso and Lai, Ying-Cheng},
  journal={Physical Review Research},
  volume={3},
  number={1},
  pages={013090},
  year={2021},
  publisher={APS}
}

@book{mattila1999geometry,
  title={Geometry of sets and measures in Euclidean spaces: fractals and rectifiability},
  author={Mattila, Pertti},
  volume={44},
  year={1999},
  publisher={Cambridge university press}
}

@techreport{kalman1970lectures,
  title={Lectures on controllability and observability},
  author={Kalman, Rudolf Emil},
  year={1970}
}

@book{abraham2012manifolds,
  title={Manifolds, tensor analysis, and applications},
  author={Abraham, Ralph and Marsden, Jerrold E and Ratiu, Tudor},
  volume={75},
  year={2012},
  publisher={Springer Science \& Business Media}
}

@article{carvalho2025finite,
  title={Finite Fractal Dimension of Uniform Attractors for Non-Autonomous Dynamical Systems with Infinite-Dimensional Symbol Space},
  author={Carvalho, Alexandre Nolasco de and Langa, Jos{\'e} Antonio and Moura, Rafael de Oliveira},
  journal={Journal of Nonlinear Science},
  volume={35},
  number={4},
  pages={70},
  year={2025},
  publisher={Springer}
}

@article{hart2024attractor,
  title={Attractor reconstruction with reservoir computers: The effect of the reservoir’s conditional Lyapunov exponents on faithful attractor reconstruction},
  author={Hart, Joseph D},
  journal={Chaos: An Interdisciplinary Journal of Nonlinear Science},
  volume={34},
  number={4},
  year={2024},
  publisher={AIP Publishing}
}

@article{mitsui2021seasonal,
  title={Seasonal prediction of Indian summer monsoon onset with echo state networks},
  author={Mitsui, Takahito and Boers, Niklas},
  journal={Environmental Research Letters},
  volume={16},
  number={7},
  pages={074024},
  year={2021},
  publisher={IOP Publishing}
}

@inproceedings{pascanu2013difficulty,
  title={On the difficulty of training recurrent neural networks},
  author={Pascanu, Razvan and Mikolov, Tomas and Bengio, Yoshua},
  booktitle={International conference on machine learning},
  pages={1310--1318},
  year={2013},
  organization={PMLR}
}

@article{jaeger2001echo,
      title={{The “echo state” approach to analysing and training recurrent neural networks-with an erratum note}},
      author={Jaeger, Herbert},
      journal={Bonn, Germany: German National Research Center for Information Technology GMD Technical Report},
      volume={148},
      number={34},
      pages={13},
      year={2001},
      publisher={Bonn}
}

@article{grigoryeva2018echo,
  title={Echo state networks are universal},
  author={Grigoryeva, Lyudmila and Ortega, Juan-Pablo},
  journal={Neural Networks},
  volume={108},
  pages={495--508},
  year={2018},
  publisher={Elsevier}
}

@article{Gonon2021FadingUniversal,
    title = {{Fading memory echo state networks are universal}},
    year = {2021},
    journal = {Neural Networks},
    author = {Gonon, Lukas and Ortega, Juan Pablo},
    month = {6},
    pages = {10--13},
    volume = {138},
    publisher = {Elsevier Ltd},
    doi = {10.1016/j.neunet.2021.01.025},
    issn = {18792782},
    pmid = {33611064},
    arxivId = {2010.12047}
}

@book{KloedenNonautonomousSystems,
  title={Nonautonomous dynamical systems},
  author={Kloeden, Peter E and Rasmussen, Martin},
  number={176},
  year={2011},
  publisher={American Mathematical Soc.}
}

@article{Ceni2020echo,
    title = {{The echo index and multistability in input-driven recurrent neural networks}},
    year = {2020},
    journal = {Physica D: Nonlinear Phenomena},
    author = {Ceni, Andrea and Ashwin, Peter and Livi, Lorenzo and Postlethwaite, Claire},
    month = {11},
    volume = {412},
    publisher = {Elsevier B.V.},
    doi = {10.1016/J.PHYSD.2020.132609},
    issn = {01672789},
    arxivId = {2001.07694}
}

@ARTICLE{Maass2002,
  author={Maass, Wolfgang and Natschläger, Thomas and Markram, Henry},
  journal={Neural Computation}, 
  title={Real-Time Computing Without Stable States: A New Framework for Neural Computation Based on Perturbations}, 
  year={2002},
  volume={14},
  number={11},
  pages={2531-2560},
  doi={10.1162/089976602760407955}}

@article{LiESN,
title = {Optimization and applications of echo state networks with leaky- integrator neurons},
journal = {Neural Networks},
volume = {20},
number = {3},
pages = {335-352},
year = {2007},
note = {Echo State Networks and Liquid State Machines},
issn = {0893-6080},
doi = {https://doi.org/10.1016/j.neunet.2007.04.016},
url = {https://www.sciencedirect.com/science/article/pii/S089360800700041X},
author = {Herbert Jaeger and Mantas Lukoševičius and Dan Popovici and Udo Siewert}
    }

@article{ceni2020interpreting,
  title={Interpreting recurrent neural networks behaviour via excitable network attractors},
  author={Ceni, Andrea and Ashwin, Peter and Livi, Lorenzo},
  journal={Cognitive Computation},
  volume={12},
  number={2},
  pages={330--356},
  year={2020},
  publisher={Springer}
}

@article{sussillo2013opening,
  title={Opening the black box: low-dimensional dynamics in high-dimensional recurrent neural networks},
  author={Sussillo, David and Barak, Omri},
  journal={Neural computation},
  volume={25},
  number={3},
  pages={626--649},
  year={2013},
  publisher={MIT Press One Rogers Street, Cambridge, MA 02142-1209, USA journals-info~…}
}

@incollection{dominey2021cortico,
  title={Cortico-Striatal Origins of Reservoir Computing, Mixed Selectivity, and Higher Cognitive Function},
  author={Dominey, Peter Ford},
  booktitle={Reservoir Computing},
  pages={29--58},
  year={2021},
  publisher={Springer}
}

@article{dominey1995complex,
  title={Complex sensory-motor sequence learning based on recurrent state representation and reinforcement learning},
  author={Dominey, Peter F},
  journal={Biological cybernetics},
  volume={73},
  number={3},
  pages={265--274},
  year={1995},
  publisher={Springer}
}

@article{lu2018attractor,
  title={Attractor reconstruction by machine learning},
  author={Lu, Zhixin and Hunt, Brian R and Ott, Edward},
  journal={Chaos: An Interdisciplinary Journal of Nonlinear Science},
  volume={28},
  number={6},
  year={2018},
  publisher={AIP Publishing}
}

@article{young1981capacity,
  title={Capacity of attractors},
  author={Young, Lai-Sang},
  journal={Ergodic Theory and Dynamical Systems},
  volume={1},
  number={3},
  pages={381--388},
  year={1981},
  publisher={Cambridge University Press}
}

@book{datseris2022nonlinear,
  title={Nonlinear dynamics: a concise introduction interlaced with code},
  author={Datseris, George and Parlitz, Ulrich},
  year={2022},
  publisher={Springer Nature}
}

@article{jaeger2002adaptive,
  title={Adaptive nonlinear system identification with echo state networks},
  author={Jaeger, Herbert},
  journal={Advances in neural information processing systems},
  volume={15},
  year={2002}
}

@inproceedings{xia2008complex,
  title={A complex echo state network for nonlinear adaptive filtering},
  author={Xia, Yili and Mandic, Danilo P and Van Hulle, Marc M and Principe, Jose C},
  booktitle={2008 IEEE Workshop on Machine Learning for Signal Processing},
  pages={404--408},
  year={2008},
  organization={IEEE}
}

@article{hewamalage2021recurrent,
  title={Recurrent neural networks for time series forecasting: Current status and future directions},
  author={Hewamalage, Hansika and Bergmeir, Christoph and Bandara, Kasun},
  journal={International Journal of Forecasting},
  volume={37},
  number={1},
  pages={388--427},
  year={2021},
  publisher={Elsevier}
}

@inproceedings{
    tallec2018can,
    title={Can recurrent neural networks warp time?},
    author={Corentin Tallec and Yann Ollivier},
    booktitle={International Conference on Learning Representations},
    year={2018},
    url={https://openreview.net/forum?id=SJcKhk-Ab},
}

@article{ceni2023transitions,
  title={Transitions in echo index and dependence on input repetitions},
  author={Ashwin, Peter and Ceni, Andrea},
  journal={Physica D: Nonlinear Phenomena},
  volume={467},
  pages={134277},
  year={2024},
  publisher={Elsevier}
}

@article{flynn2021multifunctionality,
  title={Multifunctionality in a reservoir computer},
  author={Flynn, Andrew and Tsachouridis, Vassilios A and Amann, Andreas},
  journal={Chaos: An Interdisciplinary Journal of Nonlinear Science},
  volume={31},
  number={1},
  year={2021},
  publisher={AIP Publishing}
}

@article{hart2020embedding,
  title={Embedding and approximation theorems for echo state networks},
  author={Hart, Allen and Hook, James and Dawes, Jonathan},
  journal={Neural Networks},
  volume={128},
  pages={234--247},
  year={2020},
  publisher={Elsevier}
}

@article{pecora1991driving,
  title={Driving systems with chaotic signals},
  author={Pecora, Louis M and Carroll, Thomas L},
  journal={Physical review A},
  volume={44},
  number={4},
  pages={2374},
  year={1991},
  publisher={APS}
}

@book{temam2012infinite,
  title={Infinite-dimensional dynamical systems in mechanics and physics},
  author={Temam, Roger},
  volume={68},
  year={2012},
  publisher={Springer Science \& Business Media}
}

@book{chepyzhov2002attractors,
  title={Attractors for equations of mathematical physics},
  author={Chepyzhov, Vladimir V and Vishik, Mark I},
  volume={49},
  year={2002},
  publisher={American Mathematical Soc.}
}

@article{cunha2024smoothing,
  title={Smoothing and finite-dimensionality of uniform attractors in Banach spaces},
  author={Cunha, Arthur and Carvalho, Alexandre and Cui, Hongyong and Langa, Jose},
  journal={Anais do XVI ENAMA},
  pages={133},
  year={2024}
}

@article{munkres2000topology,
  title={Topology},
  author={James, R Munkres},
  journal={Prentic Hall of India Private Limited, New delhi},
  volume={7},
  year={2000},
  publisher={Springer}
}

@book{robinson2012attractors,
  title={Attractors for infinite-dimensional non-autonomous dynamical systems},
  author={Carvalho, Alexandre and Langa, Jos{\'e} A and Robinson, James},
  volume={182},
  year={2012},
  publisher={Springer Science \& Business Media}
}

@inproceedings{mane2006dimension,
  title={On the dimension of the compact invariant sets of certain non-linear maps},
  author={Ma{\~n}{\'e}, Ricardo},
  booktitle={Dynamical Systems and Turbulence, Warwick 1980: Proceedings of a Symposium Held at the University of Warwick 1979/80},
  pages={230--242},
  year={2006},
  organization={Springer}
}

@article{hunt1992prevalence,
  title={Prevalence: a translation-invariant “almost every” on infinite-dimensional spaces},
  author={Hunt, Brian R and Sauer, Tim and Yorke, James A},
  journal={Bulletin of the American mathematical society},
  volume={27},
  number={2},
  pages={217--238},
  year={1992}
}

@article{grassberger1983measuring,
  title={Measuring the strangeness of strange attractors},
  author={Grassberger, Peter and Procaccia, Itamar},
  journal={Physica D: nonlinear phenomena},
  volume={9},
  number={1-2},
  pages={189--208},
  year={1983},
  publisher={Elsevier}
}

@article{grigoryeva2021chaos,
  title={Chaos on compact manifolds: Differentiable synchronizations beyond the {T}akens theorem},
  author={Grigoryeva, Lyudmila and Hart, Allen and Ortega, Juan-Pablo},
  journal={Physical Review E},
  volume={103},
  number={6},
  pages={062204},
  year={2021},
  publisher={APS}
}

@book{robinson2010embedding, 
    place={Cambridge}, 
    series={Cambridge Tracts in Mathematics}, 
    title={Dimensions, Embeddings, and Attractors}, 
    publisher={Cambridge University Press}, 
    author={Robinson, James C.}, year={2010}, 
    collection={Cambridge Tracts in Mathematics}
}

@article{parlitz1996generalized,
  title={Generalized synchronization, predictability, and equivalence of unidirectionally coupled dynamical systems},
  author={Kocarev, Ljupco and Parlitz, Ulrich},
  journal={Physical review letters},
  volume={76},
  number={11},
  pages={1816},
  year={1996},
  publisher={APS}
}

@article{hart2021strange,
  title={Learning strange attractors with reservoir systems},
  author={Grigoryeva, Lyudmila and Hart, Allen and Ortega, Juan-Pablo},
  journal={Nonlinearity},
  volume={36},
  number={9},
  pages={4674},
  year={2023},
  publisher={IOP Publishing}
}

@inproceedings{takens2006detecting,
  title={Detecting strange attractors in turbulence},
  author={Takens, Floris},
  booktitle={Dynamical Systems and Turbulence, Warwick 1980: proceedings of a symposium held at the University of Warwick 1979/80},
  pages={366--381},
  year={2006},
  organization={Springer}
}

@book{apostol1974mathematical,
  title={Mathematical Analysis},
  author={Apostol, T.M.},
  isbn={9780201002881},
  lccn={72114733},
  series={Addison-Wesley series in mathematics},
  url={https://books.google.co.uk/books?id=Le5QAAAAMAAJ},
  year={1974},
  publisher={Addison-Wesley}
}

@article{saint2016open,
  title={Open differentiable mappings},
  author={Saint Raymond, Jean},
  journal={Le Matematiche},
  volume={71},
  number={2},
  pages={203--214},
  year={2016}
}

@article{lorenz1963deterministic,
  title={Deterministic nonperiodic flow},
  author={Lorenz, Edward N},
  journal={Journal of atmospheric sciences},
  volume={20},
  number={2},
  pages={130--141},
  year={1963}
}

\end{document}